\documentclass[10pt]{amsart}
\usepackage{fullpage}
\usepackage{amssymb,amsmath,amsfonts,amsthm,mathrsfs}
\usepackage{color}
\definecolor{backgroundcolor}{rgb}{1,1,0.8}

\numberwithin{equation}{section}

\renewcommand{\AA}{\mathbb A}

\newcommand{\CC}{\mathbb C}

\newcommand{\FF}{\mathbb F}

\newcommand{\PP}{\mathbb P}
\newcommand{\QQ}{\mathbb Q}
\newcommand{\RR}{\mathbb R}

\newcommand{\ZZ}{\mathbb Z} 
\newcommand{\Zhat}{\widehat\ZZ}

\newcommand{\OO}{\mathcal O}
\newcommand{\calE}{\mathcal E}

\newcommand{\calH}{\mathcal H}

\newcommand{\calB}{\mathcal B}
\newcommand{\calS}{\mathcal S}
\newcommand{\calA}{\mathcal A}

\newcommand{\calU}{\mathcal U}
\newcommand{\calV}{\mathcal V}
\newcommand{\calZ}{\mathcal Z}
\newcommand{\calL}{\mathcal L}
\newcommand{\calD}{\mathcal D}
\newcommand{\calW}{\mathcal W}

 \newcommand{\bfGr}{\mathbf {Gr}}

\newcommand{\p}{\mathfrak p}

\newcommand{\scrF}{\mathscr F}

\newcommand{\norm}[1]{ \left|\!\left| #1 \right|\!\right|  }

\def\cyc{{\operatorname{cyc}}}
\def\ab{{\operatorname{ab}}}
\def\un{{\operatorname{un}}}
\def\Spec{\operatorname{Spec}} 
\def\Gal{\operatorname{Gal}}
\def\ord{\operatorname{ord}} 
\def \GL {\operatorname{GL}}  

\def \SL {\operatorname{SL}}

\def\Aut{\operatorname{Aut}} 
\def\End{\operatorname{End}}
\def\Frob{\operatorname{Frob}}

\def\Fr{\operatorname{Fr}}
\def\tr{\operatorname{tr}}
\def\Tr{\operatorname{Tr}}

\def\bbar#1{\setbox0=\hbox{$#1$}\dimen0=.2\ht0 \kern\dimen0 
\overline{\kern-\dimen0 #1}}
\newcommand{\Qbar}{{\overline{\mathbb Q}}} 
\newcommand{\Kbar}{\bbar{K}} 
 
\newcommand{\kbar}{\bbar{k}} 
\newcommand{\FFbar}{\overline{\FF}} 

\newcommand\legendre[2]{\Bigl(\frac{#1}{#2}\Bigr) }   

\newtheorem{thm}{Theorem}[section]
\newtheorem{lemma}[thm]{Lemma}

\newtheorem{prop}[thm]{Proposition}

\theoremstyle{definition}

\theoremstyle{remark}
\newtheorem{remark}[thm]{Remark}
\newtheorem{example}[thm]{Example}

\newenvironment{romanenum}{\hfill \begin{enumerate} }{\end{enumerate}}

\definecolor{webcolor}{rgb}{0.8,0,0.2}
\definecolor{webbrown}{rgb}{.6,0,0}
\usepackage[
        colorlinks,
        linkcolor=webbrown,  filecolor=webcolor,  citecolor=webbrown, 
        backref,
        pdfauthor={David Zywina}, 
        pdftitle={Hilbert's irreducibility theorem and the larger sieve},
]{hyperref}
\usepackage[alphabetic,backrefs,lite]{amsrefs} 

\begin{document}
\title[Hilbert's irreducibility theorem and the larger sieve]{Hilbert's irreducibility theorem and the larger sieve}


\author{David Zywina}
\address{Department of Mathematics, University of Pennsylvania, Philadelphia, PA 19104-6395, USA}
\email{zywina@math.upenn.edu}
\urladdr{http://www.math.upenn.edu/\~{}zywina}
\date{{\today}}

\begin{abstract}
We describe an explicit version of Hilbert's irreducibility theorem using a generalization of Gallagher's larger sieve.  We give applications to the Galois theory of random polynomials, and to the images of the adelic representation associated to elliptic curves varying in rational families.  
\end{abstract}

\maketitle

\section{Introduction}
In this paper, we are interested in quantitative versions of Hilbert's irreducibility theorem (HIT).  In \S\ref{SS:HIT intro}, we will review the classical description of HIT in terms of polynomials and give a special case of our new bounds in this setting (our most general bound can be found in \S\ref{SS:main}).    As an illustration of these bounds, we then study the fundamental example of HIT in \S\ref{SS:van Waerdan}, i.e., the Galois group of a ``random'' polynomial of degree $n$.       

A more serious application is given in \S\ref{SS:intro EC} where we discuss the Galois representations associated to the division points of an elliptic curve.  We shall start with a model of a non-isotrivial elliptic curve $E$ over a field $K=k(T_1,\ldots, T_n)$ where $k$ is a number field and the $T_i$ are independent variables.    Associated to $E$, there is a Galois representation $\rho_E\colon \Gal(\Kbar/K)\to \GL_2(\Zhat)$ describing the Galois action on the torsion points of $E(\Kbar)$.    For most $n$-tuples $t=(t_1,\ldots, t_n)\in k^n$, we obtain an elliptic curve $E_t$ over $k$ by specializing each $T_i$ with $t_i$.    For a ``random'' $t\in k^n$, we will describe the image of the corresponding Galois representation $\rho_{E_t}\colon \Gal(\kbar/k)\to \GL_2(\Zhat)$.   For $k\neq \QQ$, we will see that $\rho_{E_t}(\Gal(\kbar/k))$ agrees with the image of $\rho_E$ for most $t\in k^n$.   The case $k=\QQ$ is subtler, and we will see that $\rho_{E_t}(\Gal(\Qbar/\QQ))$ is usually a subgroup of index $r$ in $\rho_E(\Kbar/K)$ where $r$ is a certain positive integer depending on $E$.

\subsection{Hilbert's irreducibility theorem}   \label{SS:HIT intro}
Let $k$ be a number field with a fixed algebraic closure $\kbar$.   Fix a monic irreducible polynomial $f(x,T_1,\ldots, T_n) \in k(T_1,\ldots, T_n)[x]$ in the variable $x$.  To ease notation slightly, we will denote the $n$-tuple of independent variables $(T_1,\ldots, T_n)$ by $T$.     
Let $L$ be the splitting field of $f(x,T)\in k(T)[x]$ in a fixed algebraic closure $\bbar{k(T)}$.   Denote the Galois group $\Gal(L/k(T))$ by $G$.  

Now let $\Omega_f$ be the set of $t \in k^n$ for which some coefficient of $f(x,T)$ has a pole at $T=t$, or for which $f(x,t)$ is not separable.  For each $\in k^n-\Omega_f$, let $L_t$ be the splitting field of $f(x,t)\in k[x]$ in $\kbar$ and define $G_t=\Gal(L_t/k)$.   Specialization induces an inclusion $G_t \subseteq G$ which is uniquely determined up to conjugation.  We then have the following:

\begin{thm}[Hilbert's irreducibility theorem] \label{T:HIT classical}
For ``most'' points $t\in k^n-\Omega_f$, we have $G_t = G$. 
\end{thm}

Of course one needs to make the ``\emph{most}'' condition precise.  In this paper, we shall interpret this in terms of natural density.   Let $H$ be the absolute (multiplicative) height on $\PP^n(\kbar)$, see \cite{MR1745599}*{\S{}B.2} for background.   For example, if $x_0,\ldots, x_n$ belong to $\ZZ$ and satisfy $\gcd(x_0,\ldots, x_n)=1$, then $H([x_0,\cdots,x_n])=\max_{i}|x_i|$.  We shall also view $H$ as a function on $k^n=\AA^n(k)$ by using the open embedding $\AA_k^n\to \PP_k^n,\, (x_1,\ldots, x_n) \mapsto [x_1,\cdots, x_n,1]$.  For any real number $B\geq 1$, there are only finitely many $t\in k^n$ with $H(t)\leq B$.  

A precise formulation of Theorem~\ref{T:HIT classical} is then the following
\[
\lim_{B\to +\infty}  \frac{|\{ t \in k^n - \Omega_f :  H(t)\leq B,\, G_t = G \}|}{|\{ t\in k^n : H(t)\leq B\}|} = 1.
\]
Intuitively, this says that if we write down large ``random'' $t_1,\ldots, t_n\in k$, then almost surely the splitting field of the polynomial $f(x,t)$ over $k$ has Galois group $G$.  As a consequence we find that $f(x,t) \in k[x]$ is irreducible for ``most'' $t\in k^n$.   Another possible notion of ``most'' is that the theorem holds for all $t$ outside a \emph{thin} subset of $k^n$ (see \cite{MR1757192}*{\S9} or \cite{MR2363329}*{\S3} for details).   

 We will also want to consider integral versions of HIT, let $\OO_k$ be the ring of integers of $k$.  For $t=(t_1,\ldots, t_n)\in \OO_k^n$, define $\norm{t}=\max_{\sigma,i}|\sigma(t_i)|$ where $\sigma$ runs over the field embeddings $\sigma\colon k \hookrightarrow \CC$.    The following theorem, which is a consequence of the large sieve, gives essentially the best general upper bound available.   For reference, we note that there are positive constants $c_{n,k}$ and $c_{n,k}'$ such that 
\begin{equation} \label{E:comparison}
|\{ t\in \OO_k^n : \norm{t} \leq B \}| \sim c_{n,k} B^{[k:\QQ]n} \quad\quad \text{and} \quad\quad  |\{ t\in k^n : H(t) \leq B \}| \sim c_{n,k}' B^{[k:\QQ](n+1)}
\end{equation}
as $B\to +\infty$.

\begin{thm}  [Cohen, Serre] \label{T:HIT large}
With notation as above, 
\begin{align*}
  |\{ t\in \OO_k^n-\Omega_f: \norm{t} \leq B,\, G_t \neq G  \}| &\ll_{n,f,k} B^{[k:\QQ](n-1/2)} \log B \quad \text{and}\\
  |\{ t\in k^n-\Omega_f: H(t) \leq B,\, G_t \neq G  \}| &\ll_{n,f,k} B^{[k:\QQ](n+1/2)} \log B.
\end{align*}
\end{thm}
This follows from Theorems 1 and 2 of \cite[\S13]{MR1757192} (where $\log B$ can be actually be replaced with $(\log B)^{\lambda}$ for some $\lambda<1$).  The integral version with a more explicit constant can be found in \cite{MR516576}. Here is an equivalent version of Theorem~\ref{T:HIT large}:

\begin{thm}  \label{T:HIT large2}
With notation as above, let $C$ be a proper subset of $G$ that is stable under conjugation.  Then
\begin{align*}
  |\{ t\in \OO_k^n-\Omega_f: \norm{t} \leq B,\, G_t \subseteq C  \}| &\ll_{n,f,k} B^{[k:\QQ](n-1/2)} \log B \quad \text{and}\\
  |\{ t\in k^n-\Omega_f: H(t) \leq B,\, G_t \subseteq C  \}| &\ll_{n,f,k} B^{[k:\QQ](n+1/2)} \log B.
\end{align*}
\end{thm}

Theorem~\ref{T:HIT large2} follows directly from Theorem~\ref{T:HIT large}.   Let us explain the other implication; we consider only the integral case.    If $G_t\neq G$, then it must lie in some maximal subgroup $M$ of $G$.   Since our $G_t$ is only uniquely defined up to conjugation, it is less ambiguous to write $G_t\subseteq \bigcup_{g\in G} g Mg^{-1}$.   So we have
\begin{equation} \label{E:HIT large connection}
|\{ t\in \OO_k^n-\Omega_f:\norm{t}\leq B,\, G_t \neq G  \}| \leq \sum_M  \Big|\Big\{ t\in \OO_k^n-\Omega_f: \norm{t} \leq B,\, G_t \subseteq \bigcup_{g\in G} gMg^{-1}  \Big\}\Big|
\end{equation}
where the sum is over representatives of the conjugacy classes of maximal subgroups of $G.$  Define $\delta(G,M):= |\bigcup_{g\in G} gMg^{-1}|/|G|$.   By Jordan's lemma \cite{MR1997347}, we know that $\bigcup_{g\in G} gMg^{-1}$ is a proper subset of $G$ (equivalently $\delta(G,M)<1$).    Applying the bound of Theorem~\ref{T:HIT large2}  to the right hand side of (\ref{E:HIT large connection}) gives
\[
|\{ t\in \OO_k^n-\Omega_f:\norm{t}\leq B,\, G_t \neq G  \}| \ll_{n,f,k} \sum_M B^{[k:\QQ](n-1/2)} \log B.
\]
We obtain Theorem~\ref{T:HIT large} by noting that the number of representatives $M$ of maximal subgroups is $O_n(1)$.
\par

Our main abstract result is the following general bound which beats the large sieve when $|C|/|G| < 1/2$.  Its proof utilizes an extension of Gallagher's \emph{larger sieve}.  We will state a more general version of this theorem in \S\ref{SS:main} that removes the assumption that $L/k(T)$ is geometric (i.e., $L\cap \kbar=k$) and gives better control over the implicit constant.

\begin{thm} \label{T:HIT larger} 
Assume that $L/k(T)$ is geometric and let $C$ be a subset of $G$ that is stable under conjugation.   Then
\begin{align*}
  |\{ t\in \OO_k^n-\Omega_f: \norm{t} \leq B,\, G_t \subseteq C  \}|& \ll_{n,f,k} B^{[k:\QQ](n-1+|C|/|G|)} \log B\quad \text{and}\\
  |\{ t\in k^n-\Omega_f: H(t) \leq B,\, G_t \subseteq C  \}|& \ll_{n,f,k} B^{[k:\QQ](n+|C|/|G|)} \log B.
\end{align*}
\end{thm}

Arguing as before, Theorem~\ref{T:HIT larger} implies that 
\begin{equation} \label{E:raw larger}
|\{ a\in \OO_k^n-\Omega_f:\norm{t}\leq B,\, G_t \neq G  \}| \ll_{n,f,k}  B^{[k:\QQ](n-1+ \delta(G))}\log B
\end{equation}
where $\delta(G)$ is the maximum of the $\delta(G,M)$ over all maximal subgroups $M$ of $G$.  The bound (\ref{E:raw larger}) is superior to that of the large sieve if $\delta(G)<1/2$.  Unfortunately $\delta(G)\geq 1/2$ for many interesting groups (an example where (\ref{E:raw larger}) is superior is when $G$ is a $p$-group with odd $p$, since one has $\delta(G)=1/p$).

As we will see in the next section, the larger sieve can be used to deal with the \emph{small} maximal subgroups $M$ of $G$, that is, small in the sense of the quantity $\delta(G,M)$.  This leaves the larger maximal subgroups to be studied using alternate methods.

\subsection{The Galois group of a random polynomial} \label{SS:van Waerdan}
We now consider the fundamental example of Hilbert's irreducibility theorem.  Fix a positive integer $n$.  For $T=(T_1,\ldots, T_n)$, define the polynomial 
\[
f(x,T) =x^n +T_1 x^{n-1} +\cdots +  T_{n-1}x+T_n.
\] 
For $t\in\ZZ^n$, let $G_t$ be the Galois group of the splitting field of $f(x,t)$ over $\QQ$.  By numbering the roots of $f(x,t)$, we may view $G_t$ as a subgroup of $S_n$.   
Hilbert's irreducibility theorem says that $G_t=S_n$ for ``most'' choices of $t\in \ZZ^n$.    

We now consider a quantitative version.  Define the following counting function
\[
E_n(B) := |\{t\in \ZZ^n:\norm{t}\leq B,\, G_t \neq S_n \}|
\]
(recall that $\norm{t}=\max_i |t_i|$).  We will restrict ourselves to $n\geq 3$, since $n=1$ is uninteresting and it is known that $E_2(B) \sim 2B \log B$.

In 1936, van der Waerden \cite{MR1550517} gave the explicit bound
\[
E_n(B) \ll_n B^{n - \frac{c}{\log\log B}} \quad\quad \text{with $c=\frac{1}{6(n-2)}$},
\]
and further conjectured that $|E_n(B)|\ll_n B^{n-1}$ for $n>2$.   Van der Waerdan's conjecture is best possible since the polynomials $f(x,t_1,\ldots, t_{n-1},0)$ are always reducible and hence $|E_n(B)|\gg B^{n-1}$.

In 1956, Knobloch \cite{MR0080071} gave the improved bound
\[
E_n(B) \ll_n B^{n - c_n} \quad\quad \text{with $c_n=\frac{1}{18n(n!)^3}$}.
\]

In 1973, Gallagher \cite{MR0332694} used a higher dimensional large sieve to give the bound
\begin{equation} \label{E:Gallagher}
E_n(B) \ll_n B^{n-1/2} (\log B)^{1- \gamma_n}
\end{equation}
where $\{\gamma_n\}$ is a sequence of positive numbers with $\gamma_n\sim (2\pi n)^{-1/2}$.    This power of the $\log B$ can be further improved, but the large sieve is incapable of lowering the power of $B$ that occurs.  

There has been some progress for small $n$.  For any $\varepsilon>0$, one has $E_3(B)\ll_\varepsilon B^{2+\varepsilon}$ and $E_4(B)\ll_\varepsilon B^{3+\varepsilon}$ (this is due to Lefton \cite{MR550295} and Dietmann \cite{MR2271383}, respectively).    We have the following modest improvement for large $n$.

\begin{prop} \label{P:modest improvement}
For all $n$ sufficiently large, we have
\[
E_n(B) \ll_{n} B^{n- \frac{1}{2}}.
\]
\end{prop}

If instead we count those $t\in \ZZ^n$ for which $G_t$ is neither $S_n$ nor the alternating group $A_n$, then we have the following significantly stronger bound.

\begin{thm} \label{T:transitive bound}  For every $\varepsilon>0$ there is an $N$ such that
\begin{equation} \label{E: transitive bound}
|\{a\in \ZZ^n:\norm{t}\leq B,\, \text{$G_t\neq S_n$ and $G_t\neq A_n$} \}| \ll_n B^{n-1 + \varepsilon}
\end{equation}
for all $n\geq N$.
\end{thm}

\begin{remark} 
It should be noted that the condition ``$G_t \neq S_n $ and $G_t\neq A_n$'' does show up in practice.     For example, let $f(x)\in\ZZ[x]$ be a separable polynomial of degree $n\geq 5$ and let $C_f$ be the hyperelliptic curve with affine model $y^2=f(x)$.   Let $J(C_f)$ be the Jacobian of $C_f$; it is an abelian variety over $\QQ$ of dimension $2\lfloor (n-1)/2 \rfloor$.  Zarhin \cite{MR1748293} has shown that if $\Gal(f)=A_n$ or $\Gal(f)=S_n$, then $\End(J(C_f)_{\Qbar})=\ZZ$.   Theorem~\ref{T:transitive bound} thus gives an upper bound on the number of $t\in\ZZ^n$ with $\norm{t}\leq B$ for which $f(x,t)$ is not separable or $\End(J(C_{f(x,t)})_{\Qbar})\neq \ZZ.$
\end{remark}

\begin{remark}
R.~Dietmann \cite{1010.5341} has recently given a proof of Theorem~\ref{T:transitive bound} that gives superior bounds than ours.   In particular, he proves that $|\{a\in \ZZ^n:\norm{t}\leq B,\, \text{$G_t\neq S_n$ and $G_t\neq A_n$} \}| \ll_{n,\varepsilon} B^{n-1 +e(n) + \varepsilon}$ where $e(n)$ is the middle binomial coefficient $\binom{n}{\lfloor n/2 \rfloor}$.  Dietmann's techniques are not sieve theoretic; he uses Galois resolvents to reduce the question to counting integral points on certain varieties.   
\end{remark}

The first thing to note is that Theorem~\ref{T:HIT larger} by itself does not lead to an improved bound for $E_n(B)$.  Let $M_1$ be the maximal subgroup of $S_n$ that stabilizes the letter $1$.    Since $\delta(S_n,M_1)= 1- \sum_{i=0}^n (-1)^i/i!$ (this is just the proportion of elements in $S_n$ that are not derangements) we find that $\limsup_{n\to \infty} \delta(S_n) \geq 1- e^{-1}$, and in fact equality holds.  Equation (\ref{E:raw larger}) would then give the inferior bound $E_n(B)\ll_n B^{n - e^{-1} +o_n(1)}.$

Instead we shall treat $M_1$ separately.  Note that $G_t\subseteq \bigcup_{g\in G} gM_1g^{-1}$ if and only if $f(x,t)$ has a root in $\ZZ$.     The following theorem bounds the number of $t$ with $f(x,t)$ reducible.

\begin{thm}[van der Waerden \cite{MR1550517}]  \label{P:reducible}  For an integer $1\leq i \leq n/2$, we have
\[
|\{t\in \ZZ^n:\norm{t}\leq B,f(x,t) \text{ is reducible with a factor of degree $i$}\}| \ll_n \begin{cases}
      B^{n-i} & \text{ if $i< n/2$}, \\
      B^{n-i}\log B & \text{ if $i=n/2$.}
\end{cases}
\]
\end{thm}

\begin{remark}  Using van der Waerdan's theorem and counting those $t$ for which $f(x,t)$ has a root in $\ZZ$, Chela \cite{MR0148647} showed that
\[
{|\{t\in \ZZ^n:\norm{t}\leq B,\, f(x,t) \text{ is reducible}\}|} \sim  {c_n}{B^{n-1}}
\]
as $B\to +\infty$, where $c_n>0$ is an explicit constant.  
\end{remark}

Using this theorem we now need only consider those $t$ for which $f(x,t)$ is irreducible; in other words, those $t$ for which $G_t$ is a transitive subgroup of $S_n$.    Let $\mathcal{M}_n$ be the set of transitive subgroups of $S_n$ that are neither $A_n$ or $S_n$.   The following theorem of \L uczak and Pyber shows that few elements of $S_n$ belong to any of the $M\in \mathcal{M}_n$.

\begin{thm}[\L uczak-Pyber \cite{MR1476464}]  \label{T:LP} 
We have $\displaystyle\lim_{n\to \infty} \frac{|\bigcup_{M\in \mathcal{M}_n} M|}{|S_n|}= 0.$
\end{thm}

\begin{proof}[Proof of Theorem~\ref{T:transitive bound}]
From Theorem~\ref{P:reducible}, we know that
\begin{equation} \label{E: transitive bound a}
|\{t\in \ZZ^n:\norm{t}\leq B,\, G_t \text{ is a non-transitive subgroup of $S_n$}\}| \ll_n B^{n-1}.
\end{equation}
By Theorem~\ref{T:LP} there exists an $N$ such that $|\bigcup_{M\in\mathcal{M}_n} M|/|S_n| < \varepsilon$ for all $n\geq N$.  Applying Theorem~\ref{T:HIT larger} with $C=\bigcup_{M\in\mathcal{M}_n} M$ gives
\begin{equation} \label{E: transitive bound b}
|\{t\in \ZZ^n:\norm{t}\leq B,\, G_a \in \mathcal{M}_n\}| \ll_n B^{n-1 + \varepsilon} 
\end{equation}
for all $n\geq N$.  Theorem~\ref{T:transitive bound} follows by combining (\ref{E: transitive bound a}) and (\ref{E: transitive bound b}). 
\end{proof}

Thus to improve on Gallagher's bound, at least for $n$ large enough, it suffices to bound the function 
\[
E'_n(T)=|\{t\in \ZZ^n:\norm{t}\leq B,\, G_t \subseteq A_n\}|.
\]
Equivalently, bound the number of $t \in \ZZ^n$ with $\norm{t}\leq B$ for which $\Delta(t_1,\dots,t_n)$ is a square, where $\Delta(T_1,\ldots, T_n)\in k[T_1,\ldots, T_n]$ is the discriminant of $x^n +T_1 x^{n-1} +\cdots +  T_{n-1}x+T_n$.   Using the large sieve one can show that $E'_n(T)\ll_n B^{n-1/2}$ which completes the proof of Proposition~\ref{P:modest improvement}.

\begin{remark} 
In the final comments of \cite{MR1476464}, the authors claim that ${|\bigcup_{M\in \mathcal{M}_n} M|}/{|S_n|}  =O( n^{-\alpha})$ for some absolute constant $\alpha>0$.  This would imply the following  strengthening of (\ref{E: transitive bound}):
\[
|\{ t \in \ZZ^n: \norm{t}\leq B,\, G_t \neq S_n \text{ and } G_t \neq A_n\}| \ll_{n} B^{n-1+O(n^{-\alpha})}.
\]
We should also point out that an analogue of Theorem~\ref{T:LP} has recently been proven for almost simple Chevalley group over $\FF_q$ where the rank is fixed and $q\to \infty$ \cite{Fulman-Guralnick}.
\end{remark}

\subsection{Galois actions on the torsion points of elliptic curves}  \label{SS:intro EC}

\subsubsection{Serre's open image theorem}
Consider an elliptic curve $E$ defined over a field $K$.  For each positive integer $m$ relatively prime to the characteristic of $K$, let $E[m]$ be the $m$-torsion subgroup of $E(\Kbar)$.  The group $E[m]$ is non-canonically isomorphic to $(\ZZ/m\ZZ)^2$ and has a natural $\Gal(\Kbar/K)$-action which can be expressed in terms of a Galois representation 
\[
\rho_{E,m}\colon \Gal(\Kbar/K) \to \Aut(E[m]) \cong \GL_2(\ZZ/m\ZZ).
\]
If $K$ has characteristic $0$, then combining these representations together we obtain a single Galois representation
\[
\rho_{E}\colon \Gal(\Kbar/K) \to \GL_2(\Zhat)
\]
which describes the Galois action on all the torsion points of $E$ (where $\Zhat$ is the profinite completion of $\ZZ$).    The main result for these representations over number fields is the following important theorem of Serre \cite{MR0387283}.

\begin{thm}[Serre] 
Let $k$ be a number field and let $E$ be an elliptic curve over $k$ without complex multiplication.  Then $\rho_E(\Gal(\kbar/k))$ is a finite index subgroup of $\GL_2(\Zhat)$.
\end{thm}

\subsubsection{Families of elliptic curves}  \label{SS:families of EC}
Fix a number field $k$, an integer $n\geq 1$, and define the field $K:=k(T_1,\ldots, T_n)=k(T)$.   Let $E$ be an elliptic curve over the function field $K$, and assume that the $j$-invariant of $E$ does not belong to $k$.  Now choose a model for $E/K$, say, a short Weierstrass model
\[
y^2= x^3 + a(T)x + b(T).
\]
Let $\Omega$ be the set of $t\in k^n$ for which $a(T)$ and $b(T)$ have a pole at $T=t$ or for which the discriminant of the Weierstrass equation is zero at $T=t$.  Then for each $t\in k^n-\Omega$, the curve $E_t$ obtained by replacing $T$ with $t$ in our model, i.e., $y^2= x^3 + a(t)x + b(t)$, is an elliptic curve over $k$.  Our goal is to understand how the images of $\rho_{E_t}$ vary with $t\in k^n- \Omega$, and in particular to describe the image for ``most'' $t$ in terms of $E/K$.

For each integer $m\geq 1$, we define the group $\calH_E(m)= \rho_{E,m}(\Gal(\Kbar/K))$.  Specialization by $t\in k^n- \Omega$ gives an inclusion  $\rho_{E_t,m}(\Gal(\kbar/k))\subseteq \calH_E(m)$ that is determined up to conjugation.  We may thus view $\rho_{E_t}(\Gal(\kbar/k))$ as a subgroup of $\calH_E := \rho_E(\Gal(\Kbar/K))$ which again is uniquely determined up to conjugation.  Hilbert's irreducibility theorem implies that $\rho_{E_t,m}(\Gal(\kbar/k))= \calH_E(m)$ for ``most'' $t$ (where $m$ is fixed).  It also \emph{suggests} that $\rho_{E_t}(\Gal(\kbar/k))=\calH_E$ holds for most $t$.

\begin{thm} \label{T:HIT EC}
Suppose that $k\neq \QQ$.  Then 
\begin{align*}
\frac{|\{ t\in k^n-\Omega: H(t)\leq B,\, \rho_{E_t}(\Gal(\kbar/k))= \calH_E \}|}{|\{ t\in k^n : H(t)\leq B \}|} &= 1 +O\big( B^{-1/2}\log B\big) \quad \text{and}\\
\frac{|\{ t\in \OO_k^n-\Omega: \norm{t}\leq B,\, \rho_{E_t}(\Gal(\kbar/k))= \calH_E \}|}{|\{ t\in \OO_k^n : \norm{t}\leq B \}|} &= 1  +O\big( B^{-1/2}\log B\big)
\end{align*}
where the implicit constants do not depend on $B$.
\end{thm}

First observe that the choice of model is not important for this theorem.  The specializations of any two models will agree away from some closed subvariety $Z\subsetneq \AA^n_k$, and the $k$-rational points of $Z$ have density zero in $k^n$.

Secondly, it is important to note that Theorem~\ref{T:HIT EC} is \emph{not} a direct consequence of HIT (since $\calH_E$ is an infinite group).  This is well illustrated by the fact that Theorem~\ref{T:HIT EC} can fail when $k=\QQ$.\\

Let us describe why we excluded $k=\QQ$.  Recall that for a profinite group $H$, the \emph{commutator subgroup} $[H,H]$ is the smallest closed normal subgroup of $H$ for which $H/[H,H]$ is abelian. 

  Fix a $t\in \QQ^n-\Omega$, and suppose that $\rho_{E_t}(\Gal(\Qbar/\QQ))=\calH_E$.  The homomorphism $\det\circ \rho_{E_t}\colon \Gal(\Qbar/\QQ)\to \Zhat^\times$ is the cyclotomic character.   Therefore, $\rho_{E_t}(\Gal(\Qbar/\QQ^\cyc))= \calH_E \cap \SL_2(\Zhat)$ where $\QQ^\cyc$ is the cyclotomic extension of $\QQ$.   Let $\QQ^\ab$ be the maximal abelian extension of $\QQ$.  The commutator subgroup of $\rho_{E_t}(\Gal(\Qbar/\QQ))$ is $\rho_{E_t}(\Gal(\Qbar/\QQ^\ab))$, so $\rho_{E_t}(\Gal(\Qbar/\QQ^\ab))=[\calH_E,\calH_E]$.

The Kronecker-Weber theorem says that $\QQ^\cyc=\QQ^\ab$, so an equality $\rho_{E_t}(\Gal(\Qbar/\QQ))=\calH_E$ would imply that $[\calH_E,\calH_E]=\calH_E\cap \SL_2(\Zhat)$.  This relation need not hold though!  For example, with $\calH_E=\GL_2(\Zhat)$, the group $[\GL_2(\Zhat),\GL_2(\Zhat)]$ has index 2 in $\SL_2(\Zhat)$.  Our main result for $k=\QQ$ is the following.

\begin{thm} \label{T:EC Main Q}
Suppose that $k= \QQ$.  Let $r$ be the index of $[\calH_E,\calH_E]$ in $\calH_E\cap \SL_2(\Zhat)$.  Then for any $\varepsilon>0$,
\begin{align*} 
\frac{\big|\big\{ t\in \QQ^n-\Omega: H(t)\leq B,\, \big[\calH_E : \rho_{E_t}(\Gal(\Qbar/\QQ))\big] = r \big\}\big|}{|\{ t\in \QQ^n : H(t)\leq B \}|} &=1 +O(B^{-1/2+\varepsilon}) \quad\text{and}\\
\frac{\big|\big\{ t\in \ZZ^n-\Omega: \norm{t}\leq B,\, \big[\calH_E : \rho_{E_t}(\Gal(\Qbar/\QQ))\big] = r \big\}\big|}{|\{ t\in \ZZ^n : \norm{t}\leq B \}|} &=1 +O(B^{-1/2+\varepsilon})
\end{align*}
where the implicit constants do not depend on $B$.
\end{thm}

\begin{remark}
The proof of Theorem~\ref{T:EC Main Q} will actually show that $\rho_{E_t}(\Gal(\Qbar/\QQ^\ab))=[\calH_E,\calH_E]$ for ``most'' $t$.  For such $t$, $G=\rho_{E_t}(\Gal(\Qbar/\QQ))$ is a subgroup of $\GL_2(\Zhat)$ satisfying $\det(G)=\Zhat^\times$ and $G\cap \SL_2(\Zhat)= [\calH_E,\calH_E]$.   The group $G$ depends on $t$ and not necessarily on $E/K$ alone.
\end{remark}

These theorems build on several earlier results.   Much focus has been on the family $y^2=x^3+t_1 x + t_2$ with $(t_1,t_2)\in\ZZ^2$ in a growing box.    In this context, Duke \cite{MR1485897} showed that for ``most'' elliptic curve $E/\QQ$ one has $\rho_{E,\ell}(\Gal(\Qbar/\QQ))=\GL_2(\ZZ/\ell\ZZ)$ for all primes $\ell$.   Grant \cite{MR1775416} gave another proof with an asymptotic expression for those elliptic curves that do not have surjective mod $\ell$ representations for all $\ell$.

Cojocaru and Hall \cite{MR2189500} considered considered a fixed model of an elliptic curve $E$ over $\QQ(T)$ ($n=1$) with non-constant $j$-invariant.  They proved that for ``most'' specializations $t\in \QQ$, one has $\rho_{E_t,\ell}(\Gal(\Qbar/\QQ))=\GL_2(\ZZ/\ell\ZZ)$ for all $\ell\geq 17$.   This will be reproved when we generalize to higher dimensions and number fields and is essentially Theorem~\ref{T:EC Main Q}.

Building on Duke's theorem, Jones \cite{Jones-AAECASC} was able to show that $[\GL_2(\Zhat):\rho_E(\Gal(\Qbar/\QQ))]=2$ for ``most'' elliptic curve $E$ over $\QQ$ (such curves are called \emph{Serre curves} in the literature).  There has also been recent work of Cojocaru, Grant and Jones \cite{Co-Gr-Jo} studying Serre curves in one-parameter families which gives results similar to Theorem~\ref{T:EC Main Q}(i) with $n=1$; they give much stronger error terms than ours but their methods do not generalize to arbitrary number fields.

For $k\neq \QQ$, the integral point version of Theorem~\ref{T:HIT EC} for the family $y^2=x^3+t_1 x + t_2$ was proved in \cite{Zywina-Maximal}.   

The proofs in all these papers, except Grant's and \cite{Co-Gr-Jo}, uses some version of the large sieve (Grant's paper requires deep theorems of Mazur on elliptic curves over $\QQ$, and in particular do not generalize to the $k\neq \QQ$ setting).\\

A key ingredient in the proof of our theorems is an effective version of HIT applied to the representation $\rho_{E,\ell}$ for rational primes $\ell$.   
\begin{prop}  \label{P:mod l prototype}
For each rational prime $\ell\geq 17$, we have
\[
|\{  t\in \OO_k^n-\Omega: \norm{t}\leq B,\, \rho_{E_t, \ell}(\Gal(\kbar/k))\not\supseteq \SL_2(\ZZ/\ell\ZZ)  \}|\ll_{E} \ell^6 B^{[k:\QQ](n-1/2+O(1/\ell))} \log B.
\]
where the implicit constants depend only on the model for $E/K$ and the exceptional set $\Omega$.
\end{prop}

Since we are interested in the Galois action on the full torsion groups of elliptic curves (and hence with varying $\ell$) it is vital to have bounds with both good and explicit dependencies on $\ell$.   With $\ell>19$, one can use Faltings theorem (originally the Mordell conjecture) to prove that
\[
|\{  t\in \OO_k^n-\Omega: \rho_{E_t, \ell}(\Gal(\kbar/k))\not\supseteq \SL_2(\ZZ/\ell\ZZ)  \}| \ll_{E,\ell} 1.
\]
While this seems much stronger than Proposition~\ref{P:mod l prototype}, the difficulty in controlling how the implicit constant depends on $\ell$ makes it unusable for our application.

The other major ingredient in the proof of Theorem~\ref{T:HIT EC} will be an effective version of Serre's open image theorem due to Masser and W\"ustholz.    Note that even to prove a more \emph{qualitative} version of Theorem~\ref{T:HIT EC}, with the big-O term replaced with $o(1)$, we still need to use \emph{quantitative} HIT bounds.

\subsubsection{Examples}
We now give a few examples of families of elliptic curves to illustrate the theoretic results above.

\begin{example}
Let $E$ be the elliptic curve over the function field $k(j)$ defined by the Weierstrass equation 
\begin{equation} \label{E:j-invariant j}
y^2+xy=  x^3  -\frac{36}{j-1728}x-\frac{1}{j-1728}.
\end{equation}
This elliptic curve has $j$-invariant $j$, and for each $t\in k-\{0,1728\}$, specializing $j$ by $t$ gives an elliptic curve $E_t$ over $k$ with $j$-invariant $t$.    The image of $\rho_E$ is 
\[
\calH_E =\big\{ A \in \GL_2(\Zhat): \det(A)\in \chi_k(\Gal(\kbar/k)) \big\}
\]
where $\chi_k \colon \Gal(\kbar/k)\to \Zhat^\times$ is the cyclotomic character of $k$.  Note that $\calH_E=\GL_2(\Zhat)$ if and only if $k\cap \QQ^\cyc=\QQ$.    If $k\neq \QQ$, then by Theorem~\ref{T:HIT EC} we find that for ``most'' choices of $t\in k-\{0,1728\}$, the elliptic curve $E_t/k$ satisfies $\rho_{E_t}(\Gal(\kbar/k))=\calH_E$

Now consider the case $k=\QQ$.  We have $\calH_E=\GL_2(\Zhat)$ and $[\GL_2(\Zhat),\GL_2(\Zhat)]$ has index 2 in $\SL_2(\Zhat)$.  By Theorem~\ref{T:EC Main Q}, 
\[
[\GL_2(\Zhat): \rho_{E_t}(\Gal(\Qbar/\QQ))] =2
\]
for ``most'' $t\in\QQ-\{0,1728\}$.

Similar remarks hold for the elliptic curve $E$ over $k(a,b)$ given by the equation $y^2=x^3+ax+b$; it has the same monodromy group $\calH_E$.
\end{example}

\begin{example}
Let $E$ be be the elliptic curve over the function field $k(\lambda)$ given by the Weierstrass equation 
\[
y^2=  x(x-1)(x-\lambda).
\]
For simplicity assume that $k\cap \QQ^\cyc=\QQ$, so $\calH_E =\big\{ A \in \GL_2(\Zhat): A\equiv I \pmod{2}  \big\}.$
For each $t\in k-\{0,1\}$, specializing $\lambda$ by $t$ gives an elliptic curve $E_t  \colon y^2=x(x-1)(x-t)$ over $k$.  If $k\neq \QQ$, then for ``most'' choices of $t\in k-\{0,1\}$, the elliptic curve $E_t \colon y^2=x(x-1)(x-t)$ satisfies $\rho_{E_t}(\Gal(\kbar/k))=\calH_E$.

Now consider the case $k=\QQ$.    One can check that $[\calH_E, \calH_E]=\{A\in\SL_2(\Zhat): A\equiv I\pmod{4}\}$. Therefore by Theorem~\ref{T:EC Main Q} we know that for ``most'' choices of $t\in \QQ-\{0,1\}$, the elliptic curve $E_t\colon y^2=x(x-1)(x-t)$ satisfies 
\[
[\calH_E:\rho_{E_t}(\Gal(\Qbar/\QQ))]=\big[\calH_E\cap \SL_2(\Zhat): [\calH_E, \calH_E]\big]=8
\]
and hence $[\GL_2(\Zhat):\rho_{E_t}(\Gal(\Qbar/\QQ))] = 48$.
\end{example}

\begin{example}
Let $E$ be an elliptic curve over $\QQ(T)$ defined by replacing the variable $j$ in (\ref{E:j-invariant j}) with
\begin{equation} \label{E:hauptmodul relation}
j=\frac{(T^{16} + 256 T^8 + 4096)^3}{T^{32}(T^8 + 16)}. 
\end{equation}
For each $t\in \QQ-\{0\}$, we have a specialization $E_t/\QQ$ by replacing $T$ by $t$.   We claim that 
\[
[\GL_2(\Zhat):\rho_{E_t}(\Gal(\Qbar/\QQ))] = 1536
\]
for ``most'' $t\in \QQ-\{0\}$.

Let us briefly explain how this elliptic curve arises.  Define the function $h(z)=\eta(z)/\eta(4z)$ on the upper-half plane where $\eta$ is the Dedekind eta function.  Let $\Gamma$ be the group of $A \in \SL_2(\ZZ)$ for which $h(A\cdot z)=h(z)$ where $A$ acts on the upper-half plane via a linear fractional transformation.    We claim that $\Gamma$ is a congruence subgroup of $\SL_2(\ZZ)$ of index $48$ and level $32$, and the equation (\ref{E:hauptmodul relation}) holds when $T$ is replaced by $h(z)$ and $j$ is the modular $j$-function  (these claims are straightforward to show after observing that $h(z)^8$ is the Hauptmodul of $\Gamma_0(4)$).  Using that the Fourier expansion of $h(z)$ at $\infty$ has rational coefficients, one can argue that for each integer $m\geq 1$, the group $\pm\calH_E(m)$ is conjugate to the group generated by $\Gamma \bmod{m}$ and the matrices of the form $\left(\begin{smallmatrix}1 & 0 \\0 & d\end{smallmatrix}\right)$ with $d\in (\ZZ/m\ZZ)^\times$.  Some group theory then shows that $[\calH_E:\calH_E]=[\pm\calH_E:\pm\calH_E]$ has index 1536 in $\SL_2(\Zhat)$ (moreover, $[\calH_E:\calH_E]$ is of the form $\calH \times \prod_{\ell\neq 2} \SL_2(\ZZ_\ell)$ for a certain subgroup $\calH$ of index 1536 in $\SL_2(\ZZ_2)$).
\end{example}

\subsection{Overview}
We now give a quick overview of the rest of the paper.  In \S\ref{SS:main}, we state our main version of HIT.
In \S\ref{S:larger sieve}, we give an extension of the larger sieve to the setting of sieving rational or integral points; we also include a standalone application to arithmetic dynamics in \S\ref{S:dynamics}.  In \S\ref{S:larger sieve special}, we state a special form of our larger sieve that will be suitable for our application of HIT which will be proved in \S\ref{S:main proof}.  

Our general approach to finding bounds is to reduced to the one variable case; more geometrically, we have an open subvariety of $\AA_k^n$ which we will fiber by lines.   We then prove a version of HIT for each line separately, and then combine these individual bounds (it is thus vital to have uniform bounds, and this uniformity needs the equidistribution and Grassmannian calculations of \S\ref{S:equidistribution}).

Finally in \S\ref{S:EC}, we give the details of our theorems on elliptic curves stated in \S\ref{SS:intro EC}; this involves combining our quantitative HIT with an effective version of Serre's open image theorem due to Masser and W\"ustholz.

\subsection*{Notation}
For a number field $k$, let $\OO_k$ be the ring of integers and let $\Sigma_k$ be the set of non-zero prime ideals of $\OO_k$.  For each $\p\in\Sigma_k$, let $\FF_\p$ be the residue field $\OO_k/\p$ whose cardinality we denote by $N(\p)$.    The degree of $\p$ is the unique integer $\deg(\p)$ for which $N(\p)=p^{\deg(\p)}$ where $p$ is the prime lying under $\p$.  If $K/k$ is a finite Galois extension and $\p$ is unramified in $K$, then $(\p,K/k)$ will denote the Artin symbol which is a conjugacy class of $\Gal(K/k)$.   Let $k^\cyc$ and $k^\ab$ be the cyclotomic and maximal abelian extensions of $k$, respectively, in $\kbar$.   The absolute height on $\PP^n_k$ is denoted $H$.

For a finite group $G$, let $G^\sharp$ denote the set of conjugacy classes of $G$.  For a profinite group $G$, the \emph{commutator subgroup} $[G,G]$ is the smallest closed normal subgroup of $G$ for which $G/[G,G]$ is abelian.   We will always consider profinite groups with their profinite topology.

If $X$ is a scheme over a ring $R$ and we have a ring homomorphism $R\to R'$, then we denote by $X_{R'}$ the scheme $X\times_{\Spec R} \Spec R'$ over $R'$.   The homomorphism is implicit in the notation; it will frequently be one of the natural homomorphisms $k\to\kbar$, $\OO_k\to k$ and $\OO_k\to \FF_\p$.

Suppose that $f$ and $g$ are real valued functions of a real variable $x$. By $f\ll g$ (or $g\gg f$), we 
shall mean that there are positive constants $C_1$ and $C_2$ such that for all $x\geq C_1$, $|f(x)|\leq C_2 |g(x)|$.  We shall use $O(f)$ to denote an unspeciÞed function $g$ with $g\ll f$.   When needed we will indicate the dependence of the implied constants with subscripts on $\ll$ or $O$, and in the main results we will indicate the dependencies.

\subsection*{Acknowledgments}
Thanks to David Brown for several useful suggestions.
\section{Main version}

\subsection{Reinterpretation} \label{SS:reinterpretation}
It will be useful to view Hilbert's irreducibility theorem in terms of algebraic geometry.    Let $U$ be a non-empty open subvariety of $\PP^n_k$, and let
\[
\rho\colon \pi_1(U) \to G
\]
be a continuous and surjective homomorphism where $G$ is a finite group and $\pi_1(U)$ is the \emph{\'etale fundamental group} of $U$.  For every point $u \in U(k)$, we have a homomorphism
\[
\rho_u \colon \Gal(\kbar/k) = \pi_1(\Spec k) \xrightarrow{u_*} \pi(U) \xrightarrow{\rho} G
\]
by viewing $u$ as a $k$-morphism $\Spec k \to U$ and using the functoriality of $\pi_1$.

Denote the image of $\rho_u$ by $G_u$.  Note that we have suppressed the base points of our fundamental groups,  and thus the representations $\rho$ and $\rho_u$ are uniquely defined only up to an inner automorphism of $G$.   Moreover, the subgroup $G_u$ of $G$ is only defined up to conjugation; this is not a problem for us since the condition $G_u=G$ is well-defined.  We will frequently suppress base points when the choice does not matter.  \emph{Hilbert's irreducibility theorem} is then the statement that $G_u=G$ for ``most'' $u\in U(k)$.
\\

Let's describe how this version of HIT relates to the classical polynomial version described in the introduction.    Let $f(x,T_1,\ldots, T_n) \in k(T_1,\ldots, T_n)[x]$ be an irreducible polynomial.   Let $L$ be the splitting field of $f$ over $k(T_1,\ldots, T_n)$ in a fixed algebraic closure.  Let $X$ be a variety over $k$ with function field $L$.  The extension $L/k(T_1,\ldots,T_n)$ gives a dominant rational map $\pi\colon X \dashrightarrow \AA^n_k=\Spec k[T_1,\ldots, T_n]$.  By replacing $X$ with a suitable non-empty open subvariety, we have an \'etale morphism $\pi\colon X \to U$ where $U$ is an open subvariety of $\AA^n_k$.   Let $G$ be the group of automorphisms of $\pi\colon X\to U$.  The group $G$ acts faithfully on $X$ and $\pi$ induces an isomorphism $X/G\xrightarrow{\sim} U$, so the cover $\pi\colon X\to U$ gives a continuous homomorphism $\pi_1(U)\to G$.  Note that we have $G\cong \Gal(L/k(T))$.  For $u\in U(k)\subseteq k^n$, the group $G_u$ will agree with the corresponding group constructed in \S\ref{SS:HIT intro}.\\

\subsection{Uniform Hilbert's Irreducibility Theorem}  \label{SS:main}
Let $U$ be a non-empty open subvariety of $\PP^n_k$, and let
\[
\rho\colon \pi_1(U) \to G
\]
be a continuous and surjective homomorphism where $G$ is a finite group and $\pi_1(U)$ is the \'etale fundamental group of $U$.   Let $G^g$ be the image of $\pi_1(U_{\kbar})$ under $\rho$, and let $K$ be the minimal extension of $k$ in $\kbar$ for which $G^g$ is the image of $\pi_1(U_K)$.   We have a short exact sequence
\[
1\to G^g \to G \overset{\varphi}{\to} \Gal(K/k) \to 1.
\]
For each $u\in U(k)$, let $G_u$ be the image of 
\[
\Gal(\kbar/k)=\pi_1(\Spec k) \xrightarrow{u_*} \pi_1(U) \xrightarrow{\rho} G.
\]   
The subgroup $G_u$ of $G$ is uniquely defined up to conjugation.  

We define $\calU$ to be the open subscheme of $\PP^n_{\OO_k}$ that is the complement of the Zariski closure of $\PP^n_k-U$ in $\PP^n_{\OO_k}$.   The $\OO_k$-scheme $\calU$ has generic fiber $U$.  There exists a finite set $S\subseteq \Sigma_k$ such that $\rho$ factors through a homomorphism
\[
\pi_1(\calU_{\OO})\to G
\]
where $\OO$ is the ring of $S$-integers in $k$.  The main quantitative form of HIT in this paper is the following:

\begin{thm} \label{T:Main}
Let $C$ be a non-empty subset of $G$ that is stable under conjugation.  For each conjugacy class $\kappa \in \Gal(K/k)^\sharp$ define $C_\kappa = C \cap \varphi^{-1}(\kappa)$.   Define the numbers $\displaystyle \delta:= \max_{\kappa \in \Gal(K/k)^\sharp} \frac{1}{|\kappa|}\frac{|C_\kappa|}{|G^g|}$ and
\[
c:= |G^g|^2\exp\bigg(\sum_{\substack{\p\in S \\ \deg(\p)=1 \text{ and } N(\p)\geq |G^g|^2 }} \frac{\log N(\p)}{N(\p)} \bigg).
\]
\begin{romanenum}
\item   Assume further that $U$ is an open subvariety of $\AA^n_{k}.$  Then 
\[
|\{ u \in U(k)\cap \OO_k^n : \norm{u}\leq B,\, G_u\subseteq C\}| \ll_{U} c B^{[k:\QQ](n-1+\delta)} \log B.
\]
 \item  We have
\[
|\{ u \in U(k) : H(u)\leq B,\, G_u\subseteq C\}| \ll_{U} c B^{[k:\QQ](n + \delta)} \log B.  
\]
\end{romanenum}
In both cases, the implicit constant depends only on $U$ and the open embedding $U\subseteq \PP^n_k$.
\end{thm}

In the situation where $K=k$, we have $\delta=|C|/|G|$.  This is the case in Theorem~\ref{T:HIT larger} where we made the assumption that $L/k(T)$ is geometric, hence Theorem~\ref{T:HIT larger} is an easy consequence of Theorem~\ref{T:Main}.   

\begin{remark}
In applications, one might start with a representation $\rho\colon \pi_1(\calU')\to G$ where $\calU'$ is an open subscheme of $\PP_{\OO}^n$ for some ring $\OO$ of $S$-integers.   After possibly increasing $S$, the schemes $\calU'$ and $\calU_{\OO}$ will agree.  The reason for our construction of $\calU$ from $U\subseteq \PP^n_k$ is simply that our bounds can be expressed in terms of $U\subseteq \PP^n_k$ and the set $S$.
\end{remark}	

\section{The larger sieve} \label{S:larger sieve}

In this section, we give an extension of Gallagher's \emph{larger sieve} \cite{MR0291120} (it is Theorem~\ref{T:larger sieve 2} below in the case $k=\QQ$ and $n=1$).  Our versions can be used to sieve rational or integral points in $\PP^n_k$ or $\AA^n_k$, respectively.    The larger sieve tends to be very effective when we consider sets that have strict constraints on the size of their images modulo several primes $\p$.   An identical version of the sieve in the case $\PP^1_k$ can be found in \cite{EEHK}.    We will only use the integral point version in this paper, the rational point version is included for future reference.

\subsection{The larger sieve for rational points}

\begin{thm}[Larger sieve for $\PP^n(k)$]  \label{T:larger sieve}
Let $k$ be number field.  Let $\calA$ be a finite subset of $\PP^n(k)$ and $B>0$ a real number such that $H(P) \leq B$ for all $P \in \calA$.

Let $J$ be a finite set of maximal ideals of $\OO_k$.  For every $\p\in J $, let $g_\p\geq 1$ be a real number such that the reduction of $\calA$ in $\PP^n(\FF_\p)$ has cardinality at most $g_\p$.  Then
\[
 |\calA|\leq \frac{\displaystyle\sum_{\p \in J } \log N(\p) - [k:\QQ]\log(2 B^2)}{\displaystyle\sum_{\p \in J } \dfrac{\log N(\p)}{g_\p} - [k:\QQ]\log(2 B^2)}
\]
provided the denominator is positive.
\end{thm}

\begin{remark}
One can use Theorem~\ref{T:larger sieve} to sieve points on arbitrary quasi-projective varieties $V$ over $k$.  First choose an embedding $V\hookrightarrow \PP^n_k$ (so $V$ is open in a Zariski closed subvariety of $\PP^n_k$) and then give $V$ the corresponding height.   Note that the bound in Theorem~\ref{T:larger sieve} makes no direct reference to the dimension $n$.
\end{remark}

The main arithmetic input of the sieve is the following easy lemma.  It says that if two distinct points $P$ and $Q$ in $\PP^n(k)$ have the same reduction modulo several primes, then one of them must have large height.  We will write $P\equiv Q \bmod{\p}$ if the reduction of $P$ and $Q$ in $\PP^n(\FF_\p)$ agree.

\begin{lemma} \label{L:Weil height}
Let $P$ and $Q$ be distinct elements of $\PP^n(k)$. Then
\[
\sum_{\substack{\p\in\Sigma_k\\P \equiv Q \bmod{\p}}} \log N(\p) \leq [k:\QQ] \log\big( 2 H(P)H(Q) \big).
\]
\end{lemma}
\begin{proof}
Choose coordinates $a_i,b_j \in k$ such that $P=[a_0,\ldots,a_n]$ and $Q=[b_0,\ldots, b_n]$.
Now fix a prime ideal $\p\in\Sigma_k$ such that $P\equiv Q \bmod{\p}$. We claim that:
\begin{equation} \label{E:Weil height}
1\leq \min_{i\neq j} \ord_\p(a_ib_j-a_j b_i) - \min_{i} \ord_\p(a_i) - \min_{i} \ord_\p(b_i).
\end{equation}
Note that the right hand side of (\ref{E:Weil height}) does not depend on the initial choice of coordinates.  So without loss of generality, we may assume that $\min_{i} \ord_\p(a_i)=\min_{i} \ord_\p(b_i)=0$. Under this assumption, $P\equiv Q \bmod{\p}$ is equivalent to $\min_{i\neq j}\ord_\p(a_ib_j-a_j b_i) \geq 1$, and the claim follows.

By (\ref{E:Weil height}), we have
\begin{align} \label{E:local height bound}
\sum_{\substack{\p\in\Sigma_k\\P\equiv Q \bmod{\p}}} \log N(\p) \leq
\sum_{\p\in\Sigma_k} \min_{i\neq j} \ord_\p(a_ib_j-a_j b_i) \log N(\p) &- \sum_{\p\in\Sigma_k} \min_{i} \ord_\p(a_i) \log N(\p) \\
\notag &- \sum_{\p\in\Sigma_k} \min_{i} \ord_\p(b_i)\log N(\p).
\end{align}

Let $\Sigma_k^\infty$ be the set of archimedean places of $k$.  For each $v\in \Sigma_k^\infty$, let $|\!\cdot\!|_v$ be the extension of the usual absolute value on $\RR$ to the completion $k_v$.   Rewriting (\ref{E:local height bound}) in terms of heights gives
\begin{align*}
\frac{1}{[k:\QQ]} \sum_{\substack{\p\in\Sigma_k\\P\equiv Q \bmod{\p}}} \log N(\p) \leq &
\log H(P) + \log H(Q) - \log H([a_ib_j-a_jb_i]) \\
& + \sum_{v\in \Sigma_k^\infty} \frac{[k_v:\RR]}{[k:\QQ]} \log\Big(\frac{\max_{i\neq j} |a_i b_j -a_j b_i|_v}{\max_i |a_i|_v \cdot \max_i |b_i|_v} \Big).
\end{align*}
Using $H \geq 1$ and the triangle inequality, we have
\begin{align*}
\frac{1}{[k:\QQ]} \sum_{\p\in\Sigma_k, \, P\equiv Q \bmod{\p}} \log N(\p) &\leq
\log H(P) + \log H(Q) + {\sum}_{v\in \Sigma_k^\infty} \frac{[k_v:\RR]}{[k:\QQ]} \log 2\\
&= \log H(P) + \log H(Q) +  \log 2. \qedhere
\end{align*}
\end{proof}

\begin{proof}[Proof of Theorem~\ref{T:larger sieve}]
Fix a prime ideal $\p \in J $.   For each $c\in \PP^n(\FF_\p)$, let $Z(c,\p)$ be the number of elements in $\calA$ whose reduction in $\PP^n(\FF_\p)$ is equal to $c$.  By the Cauchy-Schwartz inequality and our assumption on the cardinality of $\calA$ modulo $\p$, we have the following inequality:
\begin{align*}
 \frac{|\calA|^2}{g_\p} = \frac{1}{g_\p} \Big(\sum_{c\in \PP^n(\FF_\p)} Z(c,\p) \Big)^2 & \leq \frac{1}{g_\p} \Big(g_\p \sum_{c\in \PP^n(\FF_\p)} Z(c,\p)^2 \Big)\\
&=\sum_{\substack{P,Q\in \calA \\ P\equiv Q \bmod{\p}}} 1 = |\calA| + \sum_{\substack{P,Q\in \calA, \, P\neq Q \\ P\equiv Q \bmod{\p}}} 1.
\end{align*}
Multiplying by $\log N(\p)$ and summing over all $\p \in J $ gives the following:
\begin{align*}
|\calA|^2 \sum_{\p\in J } \frac{\log N(\p)}{g_\p} & \leq  \sum_{\p\in J } \log N(\p) \Big(|\calA| + \sum_{\substack{P,Q\in \calA, \, P\neq Q \\ P\equiv Q \bmod{\p}}} 1 \Big)\\
&= |\calA| \sum_{\p\in J } \log N(\p) + \sum_{P,Q\in \calA, \, P\neq Q} \Big(\sum_{\substack{\p \in J  \\ P\equiv Q \bmod{\p}}}\log N(\p) \Big).
\end{align*}
By Lemma~\ref{L:Weil height}, we have
\begin{align*}
|\calA|^2 \sum_{\p\in J } \frac{\log N(\p)}{g_\p}  &\leq   |\calA| \sum_{\p\in J } \log N(\p) + \sum_{P,Q\in \calA, \, P\neq Q} [k:\QQ] \log(2H(P)H(Q)) 
\intertext{and by our choice of $B$,}
|\calA|^2 \sum_{\p\in J } \frac{\log N(\p)}{g_\p} & \leq  |\calA| \sum_{\p\in J } \log N(\p) + (|\calA|^2-|\calA|) [k:\QQ]\log(2B^2).
\end{align*}
After cancelling both sides by $|\calA|$ (the theorem is trivial if $|\calA|=0$), the theorem is immediate.
\end{proof}

\subsection{The larger sieve for integral points}

\begin{thm}[Larger sieve for $\OO_k^n$]  \label{T:larger sieve 2}
Let $k$ be number field.  Let $\calA$ be a finite subset of $\OO_k^n$ and $B>0$ a real number such that $\norm{P-Q} \leq B$ for all $P,Q \in \calA$.

Let $J$ be a finite set of maximal ideals of $\OO_k$.  For every $\p\in J $, let $g_\p\geq 1$ be a real number such that the reduction of $\calA$ in $\FF_\p^n$ has cardinality at most $g_\p$.  Then
\[
 |\calA|\leq \frac{\displaystyle\sum_{\p \in J } \log N(\p) - [k:\QQ]\log B}{\displaystyle\sum_{\p \in J } \dfrac{\log N(\p)}{g_\p} - [k:\QQ]\log B}
\]
provided the denominator is positive.
\end{thm}

\begin{lemma} \label{L:Weil height 2}
Let $P$ and $Q$ be distinct elements of $\OO_k^n$. Then
\[
\sum_{\substack{\p\in\Sigma_k\\P \equiv Q \bmod{\p}}} \log N(\p) \leq [k:\QQ] \log \norm{P-Q}.
\]
\end{lemma}
\begin{proof}
If $\p\in\Sigma_k$ is a prime ideal such that $P\equiv Q \bmod{\p}$, then 
\begin{equation*} 
 \min_{i} \ord_\p(a_i-b_i)\geq 1
\end{equation*}
where $P=(a_1,\dots,a_n)$ and $Q=(b_1,\dots,b_n)$.  Therefore, we have
\begin{align*}
\frac{1}{[k:\QQ]} \sum_{\substack{\p\in\Sigma_k\\P\equiv Q \bmod{\p}}} \log N(\p) &\leq
\frac{1}{[k:\QQ]} \sum_{\p\in\Sigma_k} \min_{i} \ord_\p(a_i- b_i) \log N(\p)  \\
 & = \sum_{v\in \Sigma_k^\infty} \frac{[k_v:\RR]}{[k:\QQ]} \log\big( \max_i |a_i-b_i|_v \big) - \log H([P-Q])\\
& \leq \sum_{v\in\Sigma_k^\infty} \frac{[k_v:\RR]}{[k:\QQ]} \log \norm{P-Q} - \log H([P-Q])\\
 & = \log \norm{P-Q} - \log H([P-Q])
\end{align*}
where $[P-Q]$ is the image of $P-Q$ in $\PP^{n-1}(k)$.  We obtain the desired inequality by noting that $H([P-Q])\geq 1$.
\end{proof}

\begin{proof}[Proof of Theorem~\ref{T:larger sieve 2}]
The proof is identical to that of Theorem~\ref{T:larger sieve}, the main difference being that we use Lemma~\ref{L:Weil height 2} in place of Lemma~\ref{L:Weil height}.
\end{proof}

\subsection{Interlude: orbits modulo $\mathfrak{p}$} \label{S:dynamics}

In this section (which is independent of the rest of the paper), we consider a problem of arithmetic dynamics studied by Silverman \cite{MR2448661}, and then by Akbary and Ghioca \cite{MR2549537}.  This quick application of our larger sieve gives a good illustration of how Theorem~\ref{T:larger sieve 2} can be used to sieve points on general quasi-projective varieties.    It is also significantly easier that our main application (Theorem~\ref{T:Main}) which requires a more elaborate proof.

Let $V$ be a quasi-projective variety defined over a number field $k$.  Fix a morphism $\varphi \colon V \to V$ and a point $P\in V(k)$.    Suppose that the forward $\varphi$-orbit 
\[
\OO_\varphi(P):=\{P, \varphi(P),\varphi^2(P), \varphi^3(P),\dots\}
\]
is infinite.  
Choose a model of $V$ and $\varphi$ over the ring of integers of $k$.  Then for all but finitely many non-zero prime ideals $\p$ of $\OO_K$,  we can (by abuse of notation) consider the reduction 
\[
\varphi_\p \colon V(\FF_\p) \to V(\FF_\p)
\]
 and the reduction $P_\p \in V(\FF_\p)$ of the point $P$.   
We define $m_\p(\varphi,P)$ to be the cardinality of the forward $\varphi_\p$-orbit 
\[
\OO_{\varphi_\p}(P_\p) := \{P_\p, \varphi_\p(P_\p),\varphi_\p^2(P_\p), \dots\}.
\]
For the finite number of excluded primes, we simply define $m_\p(\varphi,P)=+\infty$.  The choice of model for $V$ and $\varphi$ is not important for our applications since a different choice would change only finitely many of the values $m_\p(\varphi,P)$.

Since $V$ is quasi-projective, we may choose an embedding $V \subseteq \PP^n_k$ defined over $k$ (so $V$ is open in a closed subvariety of $\PP^n_k$).   Using this embedding, we equip $V$ with the height $H$ of $\PP^n_k$; it will be convenient to work with the \emph{logarithmic height} on $\PP^n_k$, i.e., $h = \log \circ H$.  By \cite{MR2448661}*{Proposition~4}, there are numbers $d>1$ and $c\geq 0$ such that $h(\varphi^i(P)) \leq d^i (h(P)+c)$ holds for all integers $i \geq 0$.

\begin{thm} \label{T:interlude}
For any $\varepsilon<1/\log d$, the set
\[
\big\{ \p \in \Sigma_K : m_\p(\varphi,P) \geq \varepsilon \log N(\p) \big\}
\]
has natural density $1$. 
\end{thm}

In \cite{MR2549537}, Akbary and Ghioca define the \emph{degree} $\deg(\varphi)$ of the morphism $\varphi$.  If $\deg(\varphi)>1$, then we can choose $d=\deg(\varphi)$ above.  Theorem~\ref{T:interlude} is then the same as Theorem 1.1(i) of \cite{MR2549537}.    If $\deg(\varphi)=1$, then \cite{MR2549537} gives a stronger bound which also follows from the larger sieve.

This theorem is a slight improvement over \cite{MR2448661}*{Theorem~3}, where it is shown that for each $\lambda < 1$, the set $\{\p : m_\p(\varphi,P) \geq (\log N(\p))^\lambda\}$ has analytic density $1$.   The bound $m_\p(\varphi,P) \geq \varepsilon \log N(\p)$ is likely far from optimal.  In fact, one expects to be able to replace $\log N(\p)$ by an appropriate power of $N(\p)$ (see \cite{MR2448661}*{\S6} for details).  

\begin{proof}[Proof of Theorem~\ref{T:interlude}]
Since $\varepsilon<1/\log d$, we can choose constants $0<\alpha<1$ and $C>1$ such that $(1+C^{-1})\varepsilon < \alpha/\log d$.  Define the function $g(x):=\varepsilon \log x$ and the set 
\[
\mathscr{S}(x):= \{ \p : N(\p)\leq x, \; m_\p(\varphi,P) \leq g(x)\}.
\] 
It suffices to show that $|\mathscr{S}(x)|=o(x/\log x)$ as $x\to +\infty$.

Define the set
\[
\calA(x) = \{ Q \in \OO_\varphi(P) : h(Q) \leq x^\alpha \}.
\]
The number of $i\geq 0$ that satisfy $d^i (h(P)+c)\leq x^\alpha$ is $\frac{\alpha}{\log d} \log x + O(1)$, so using this and the assumption  $|\OO_\varphi(P)|=\infty$ we have 
\[
|\calA(x)| \geq \frac{\alpha}{\log d} \log x + O(1). 
\]

We now find an upper bound for $|\calA(x)|$ using the larger sieve.  For each $\p \in \mathscr{S}(x)$, the reduction of $\calA(x)$ modulo $\p$ lies in $\OO_{\varphi_\p}(P_\p)$ which has cardinality at most $g(x)$.  Define $L:=\sum_{\p\in \mathscr{S}(x)} \log N(\p)$ and $\calB:= [k:\QQ]\log\big(2(e^{x^\alpha})^2\big)=[k:\QQ](2x^\alpha +\log 2).$   Assume that  $L - g(x) \geq C g(x)\calB$ holds.      Then by Theorem \ref{T:larger sieve}, we have
\[
|\calA(x)|\leq \frac{L-\calB}{L/g(x)-\calB} = g(x) + \frac{g(x)^2\calB-g(x)\calB}{L-g(x)\calB}  
\]
(from our assumption, we have $L/g(X)-\calB \geq (C-1)\calB+1>0$).   Therefore,
\[
|\calA(x)|\leq  g(x) + \frac{g(x)^2\calB-g(x)\calB}{L-g(x)\calB} \leq  g(x) + \frac{g(x)^2\calB-g(x)\calB}{Cg(x)\calB} = (1+C^{-1})g(x)+O(1).   
\]
and so $|\calA(x)| \leq (1+C^{-1})\varepsilon \log x + O(1)$.

Since $(1+C^{-1})\varepsilon < \alpha/\log d$, our lower and upper bounds for $|\calA(x)|$ are contradictory for all sufficiently large $x$.  Therefore, we must have $L - g(x) \leq C g(x)\calB$.  Thus
\[
\sum_{\p \in \mathscr{S}(x)} \log N(\p) \leq C\varepsilon [k:\QQ] (\log x)(2x^\alpha +\log 2) +\varepsilon\log x \ll x^\alpha \log x.
\]
Using partial summation, this implies that $|\mathscr{S}(x)| \ll x^\alpha$.  In particular, $|\mathscr{S}(x)|=o(x/\log x)$.
\end{proof}

\section{Special case of larger sieve} \label{S:larger sieve special}

In this section we deduce some bounds from our larger sieve.  We will of course apply them later to obtain bounds for Hilbert's Irreducibility Theorem, but to simplify the exposition we will keep this application separate.

\begin{prop} \label{P:specialized larger sieve}
Let $k$ be a number field and let $S$ a finite subset of $\Sigma_k$.
\begin{romanenum}
\item \emph{(Rational points)}
  Let $\calA$ a subset of $\PP^n(k)$ such that $H(P) \leq B$ for all $P\in\calA$.
Suppose that for each $\p\in \Sigma_k-S$, the cardinality of the image of $\calA$ under the reduction map $\PP^n(k)\to \PP^n(\FF_\p)$ is at most $g_\p$ where
\[
 g_\p \leq \delta \big( N(\p) +  D N(\p)^{1/2}\big)
\]
for some constants $0<\delta \leq 1$ and $D\geq 1$.  Then
\[
|\calA| \ll_k D^2 \exp\Bigg(\sum_{{\p\in S \text{ with }\deg(\p)=1 \text{ and }N(\p)\geq D^2}} \frac{\log N(\p)}{N(\p)}\bigg)  B^{2[k:\QQ]\delta}.
\]
\item \emph{(Integral points)}
 Let $\calA$ a subset of $\OO_k^n$ such that $\norm{P-Q} \leq B$ for all $P,Q\in \calA$.  Suppose that for each $\p\in \Sigma_k-S$, the cardinality of the image of $\calA$ under the reduction map $\OO^n_k\to \FF_\p^n$ is at most $g_\p$ where
\[
 g_\p \leq \delta \big( N(\p) +  D N(\p)^{1/2}\big)
\]
for some constants $0<\delta \leq 1$ and $D\geq 1$.  Then 
\[
|\calA| \ll_k D^2 \exp\bigg(\sum_{\p\in S \text{ with }\deg(\p)=1 \text{ and }N(\p)\geq D^2} \frac{\log N(\p)}{N(\p)}\bigg)  B^{[k:\QQ]\delta}.
\]
\end{romanenum}
\end{prop}

\begin{remark} 
\begin{romanenum}
\item
The condition on $g_\p$ is quite common when $n=1$ where it implies that the proportion of elements of $\PP^1(\FF_{\p})$ (or $\AA^1(\FF_{\p})$) that belong to $\mathcal{A} \bmod{\p}$ is at most $\delta$.
\item
 In Corollary 19 and 20 of \cite{EEHK}, there are similiar results under the much stronger hypothesis that $g_\p \leq C N(\p)^{\alpha}$ where $C>0$ and $0\leq \alpha<1$ are constants (they state it only for subset $\calA$ of $\PP^1(k)$ but everything easily generalizes to our setting).   They use this stronger hypothesis to obtains explicit bounds for $|\calA|$ that are polynomial in $\log B$.
\end{romanenum}
\end{remark}

\subsection{Analytic bounds} \label{SS:analytic bounds}

\begin{lemma} \label{L:PNT input} For a number field $k$ and a real number $x\geq 1$,
\[
\sum_{\p\in \Sigma_k,\, N(\p)\leq x} \frac{\log N(\p)}{N(\p)} = \log x +O_k(1)
\quad\text{and}\quad \sum_{\p\in \Sigma_k,\, N(\p) \geq x} \frac{\log N(\p)}{N(\p)^{3/2}} \ll_k \frac{1}{x^{1/2}}.
\]
\end{lemma}
\begin{proof}
By partial summation (\cite{MR2376618}*{Theorem~2.1.1}), we have
\[
 \sum_{\p\in \Sigma_k,\, N(\p)\leq x} \frac{\log N(\p)}{N(\p)} = \frac{\psi_k(x)}{x} + \int_2^x \frac{\psi_k(t)}{t^2} dt
\]
where $\psi_k(x)=\sum_{\p\in\Sigma_k,\, N(\p)\leq x} \log N(\p)$.  By the prime number theorem (with a worked out error term), we have $\psi_k(x)= x + O_k\big(x/(\log x)^A\big)$ for some constant $A>1$.  Therefore,
\[
 \sum_{\p\in \Sigma_k,\, N(\p)\leq x} \frac{\log N(\p)}{N(\p)} = O_k(1) + \int_2^x \frac{dt}{t}  + O_k\Big(\int_2^x \frac{dt}{t(\log t)^A} \Big) = \log x + O_k(1).
\]
The second expression is proven in a similiar fashion.
\end{proof}

\begin{lemma} \label{L:need PNT}
Let $k$ be a number field and fix a constant $D\geq 1$.  Then
\[
\sum_{D^2\leq N(\p) \leq x} \frac{\log N(\p)}{ N(\p)+ D N(\p)^{1/2}} \geq \log(x)-\log(D^2) -\alpha_k 
\]
where $\alpha_k\geq 0$ is a constant depending only on $k$.
\end{lemma}
\begin{proof}
For each prime $\p \in \Sigma_k$, we have
\[
 \frac{\log N(\p)}{N(p)+D N(\p)^{1/2}} = \frac{\log N(\p)}{N(\p)}\frac{1}{1 +D/N(\p)^{1/2}} \geq \frac{\log N(\p)}{N(\p)}\left( 1 - \frac{D}{N(\p)^{1/2}}\right).
\]
So by summing over all $\p$ with $D^2\leq N(\p)\leq x$ and using Lemma~\ref{L:PNT input}, we obtain
\begin{align*}
\sum_{D^2\leq N(\p) \leq x} \frac{\log N(\p)}{N(p)+D N(\p)^{1/2}} &\geq \sum_{D^2\leq N(\p) \leq x} \frac{\log N(\p)}{N(p)} - D \sum_{ N(\p) \geq D^2} \frac{\log N(\p)}{N(p)^{3/2}}\\
& = (\log x - \log (D^2) +O_k(1) ) + D \cdot O_k(1/(D^2)^{1/2})\\
& = \log x - \log(D^2) + O_k(1). \qedhere
\end{align*}
\end{proof}

\begin{lemma} \label{L:sieving work}
 Let $k$ be a number field and $S$ a finite subset of $\Sigma_k$.  For each $\p\in \Sigma_k-S$, fix a positive integer
$g_\p$ such that
\[
 g_\p \leq \delta \big( N(\p) + D N(\p)^{1/2} \big)
\]
where $0<\delta\leq 1$ and $D\geq 1$ are constants. Let $B\geq 1$ be any real number.  

By setting
\[
 x:= \beta_k D^2 \exp\bigg(\sum_{\p\in S,\, N(\p)\geq D^2} \frac{\log N(\p)}{N(\p)}\bigg) e^\delta B^{[k:\QQ]\delta}
\]
where $\beta_k\geq 1$ is a certain constant depending only on $k$, we obtain the bound
\begin{equation}\label{E:explicit larger inequality}
\frac{\displaystyle\sum_{\substack{\p \in \Sigma_k -S\\ D^2 \leq N(\p) \leq x}} \log N(\p) - [k:\QQ] \log B}{ \displaystyle\sum_{\substack{\p \in \Sigma_k -S\\ D^2 \leq N(\p) \leq x}} \frac{\log N(\p)}{g_\p} - [k:\QQ] \log B} \ll_k D^2 \exp\bigg(\sum_{\p\in S,\, N(\p)\geq D^2} \frac{\log N(\p)}{N(\p)}\bigg)  B^{[k:\QQ]\delta}
\end{equation}
and the denominator of (\ref{E:explicit larger inequality}) is positive.
\end{lemma}

\begin{proof}
Using the given bound on $g_\p$ and Lemma~\ref{L:need PNT} we have:
\begin{align}
\notag \sum_{\substack{\p \in \Sigma_k -S\\ D^2 \leq N(\p) \leq x}} \frac{\log  N(\p)}{g_\p} &
\geq \delta^{-1}\bigg(\sum_{\substack{\p \in \Sigma_k\\ D^2 \leq N(\p) \leq x}} \frac{\log  N(\p)}{N(\p) + D N(\p)^{1/2}}
-\sum_{\p\in S,\, N(\p)\geq D^2} \frac{\log  N(\p)}{N(\p)}\bigg)\\
\label{E:right x} & \geq \delta^{-1} \Big( \log(x) - \log(D^2)-\alpha_k - \sum_{\p \in S, \, N(\p)\geq D^2}\frac{\log N(\p)}{N(\p)}\Big).
\end{align}
Define $\beta_k:= e^{\alpha_k}$.  With our choice of $x$ we find that the expression (\ref{E:right x}) is equal to $1+[k:\QQ]\log B$, and thus 
\[
 \sum_{\substack{\p \in \Sigma_k -S\\ D^2 \leq N(\p) \leq x}} \frac{\log  N(\p)}{g_\p} -[k:\QQ]\log B \geq 1.
\]
So the denominator (and hence also the numerator) of the expression in (\ref{E:explicit larger inequality}) is at least $1$.  Thus the left hand side of (\ref{E:explicit larger inequality}) is bounded by $\sum_{N(\p)\leq x} \log N(\p) \ll_k x$.  The lemma follows by once again using our specific choice of $x$.
\end{proof}

\subsection{Proof of Proposition~\ref{P:specialized larger sieve}}
We first consider part (i). Let $J$ be the set of $\p\in \Sigma_k-S$ such that $D^2\leq N(\p)\leq x$, where $x$ is a real number to be chosen later.  By the larger sieve (Theorem~\ref{T:larger sieve}), we have the bound
\begin{equation}\label{E:sieving final step}
 |\calA|\leq \frac{\displaystyle\sum_{\p \in J } \log N(\p) - [k:\QQ]\log(2 B^2)}{\displaystyle\sum_{\p \in J } \dfrac{\log N(\p)}{g_\p} - [k:\QQ]\log(2 B^2)}
\end{equation}
provided the denominator is positive.  

Choosing $x$ as in Lemma~\ref{L:sieving work} (with $B$ replaced by $2B^2$), we find the that denominator is in fact positive.  Moreover, Lemma~\ref{L:sieving work} now tells us that $|\calA|\ll_k D^2 \exp\bigg(\substack{\sum\\\p\in S,\, N(\p)\geq D^2} \frac{\log N(\p)}{N(\p)}\bigg)  B^{2[k:\QQ]\delta}$.  Finally, we need only restrict to those $\p\in S$ with $\deg(\p)=1$ since $\sum_{\p\in\Sigma_k,\, \deg(\p)\geq 2} \frac{\log N(\p)}{N(\p)} \ll_k 1$.

Part (ii) is proven in a similiar manner; the main difference being that we use Theorem~\ref{T:larger sieve 2} instead of Theorem~\ref{T:larger sieve}.

\section{Equidistribution} \label{S:equidistribution}

In this section, we consider the equidistribution of Frobenius conjugacy classes coming from curves (and in particular lines) over finite fields.  In \S\ref{SS:Cheb finite}, we recall bounds resulting from the Grothendieck-Lefschetz trace formula and Deligne's completion of the Weil conjectures.  We will later apply these results to projective spaces by first fibering by many rational lines; it will thus be vital that our bounds are uniform.

\subsection{Chebotarev for curves over finite fields} \label{SS:Cheb finite}

Let $X$ be a smooth, projective, geometrically integral curve of genus $g$  defined over a finite field $\FF_q$ with $q$ elements.  Let $U$ be a non-empty open affine subvariety of $X$.  For each $u\in U(\FF_q)$,  the homomorphism $\Gal(\FFbar_q/\FF_q) \xrightarrow{u_*} \pi_1(U)$ is determined by the value it takes on the $q$-th power Frobenius automorphism $\Frob_q$ of $\FFbar_q$; this gives a conjugacy class $\Frob_u$ of $\pi_1(U)$.

Fix a finite group $G$ and a surjective continuous homomorphism
\[
\rho\colon \pi_1(U) \to G.
\]
Let $G^g$ denote the image of the geometric fundamental group $\pi_1(U_{\FFbar_q})$ under $\rho$.   We then have a natural exact sequence
\[
1 \to G^g \to G \stackrel{\varphi}{\to} \Gal(\FF_{q^d}/\FF_q)=:\Gamma \to 1.
\]
where $\varphi(\rho(\Frob_u))=\{\Frob_q\}$ for all $u\in U(\FF_q)$.  Assume further that the corresponding representation $\pi_1(U_{\FFbar_q})\to G$ is \emph{tamely ramified} at all the points of $(X-U)(\FFbar_q)$.

\begin{prop} \label{P:Chebotarev positive char}
With notation as above, let $C$ be a subset of $\varphi^{-1}(\Frob_q)$ that is stable under conjugation by $G$.  Then
\[
\bigg| |\{ u \in U(\FF_q): \rho(\Frob_u) \subseteq C \}| - \frac{|C|}{|G^g|} |U(\FF_q)| \bigg| 
\leq |C|^{1/2} (1-|G^g|^{-1})^{1/2}  (2g-2+\#(X-U)(\FFbar_q)) q^{1/2}.
\] 
\end{prop}
\begin{proof}(Sketch)  We follow the outline of Kowalski in \cite{MR2240230}*{Theorem~1} adding more details concerning the bounds where appropriate.  Let $M=\#(X-U)(\FFbar_q)$.

Fix a prime $\ell$ that does not divide $q$.  Let $\widehat{G}$ and $\widehat{\Gamma}$ be the set of $\Qbar_\ell$-valued irreducible characters of $G$ and $\Gamma$, respectively (i.e., those coming from finite dimensional linear representations over $\Qbar_\ell$).  Composition by $\varphi$ induces an injective $\widehat{\Gamma}\hookrightarrow \widehat{G}$ which we will sometimes view as an inclusion.  Let $\delta_C \colon G \to \{0,1\}$ be the characteristic function of $C$, which we may write in the form 
\[
\delta_C(g) = \sum_{\chi \in \widehat{G}} c_\chi \chi(g)
\]
where $c_\chi := \frac{1}{|G|}\sum_{g\in C} \overline{\chi(g)}.$  The quantity we are trying to estimate then becomes
\[
|\{ u \in U(\FF_q): \rho(\Frob_u) \subseteq C \}| = \sum_{u\in U(\FF_q)} \delta_C(\rho(\Frob_u)) = \sum_{\chi \in \widehat{G}} c_\chi \sum_{u\in U(\FF_q)}\chi(\rho(\Frob_u)).
\]
We first consider the contribution coming from those $\chi$ that arise from a character of $\Gamma$.  So
\begin{align*}
\sum_{\psi \in \widehat{\Gamma}} c_\psi \sum_{u\in U(\FF_q)}\psi(\varphi(\rho(\Frob_u)))
& = \sum_{\psi \in \widehat{\Gamma}} \frac{1}{|G|}\sum_{g\in C} \overline{\psi(\varphi(g))} \sum_{u\in U(\FF_q)}\psi(\varphi(\rho(\Frob_u)))\\
& =  \frac{1}{|G|}\sum_{g\in C} \sum_{u\in U(\FF_q)} \sum_{\psi \in \widehat{\Gamma}}\overline{\psi(\Frob_q)} \psi(\Frob_q)\\
\intertext{where the last line uses our assumption $\varphi(C)=\{\Frob_q\}$. Since all the characters of $\Gamma$ are one dimensional, we have}
\sum_{\psi \in \widehat{\Gamma}} c_\psi \sum_{u\in U(\FF_q)}\psi(\varphi(\rho(\Frob_u))) & = \frac{|\widehat{\Gamma}||C|}{|G|} |U(\FF_q)| = \frac{|C|}{|G^g|} |U(\FF_q)|;
\end{align*}
this is the ``main term'' of our estimate.  By the Cauchy-Schwarz inequality
\begin{align}
\label{E:Chebotarev character setup}
&\Big| |\{ u \in U(\FF_q): \rho(\Frob_u) \subseteq C \}| - \frac{|C|}{|G^g|} |U(\FF_q)| \Big|\\ 
\notag =& \Big| \sum_{\chi \in \widehat{G}-\widehat{\Gamma}} c_\chi \sum_{u\in U(\FF_q)}\chi(\rho(\Frob_u))\Big|\\
\notag  \leq & \Big(\sum_{\chi\in\widehat{G}} |c_\chi|^2 \Big)^{1/2} \Big( \sum_{\chi\in\widehat{G}-\widehat{\Gamma}}\Big|\sum_{u\in U(\FF_q)}\chi(\rho(\Frob_u))\Big|^2\Big)^{1/2} = \frac{|C|^{1/2}}{|G|^{1/2}} \Big(\sum_{\chi\in\widehat{G}-\widehat{\Gamma}}\Big|\sum_{u\in U(\FF_q)}\chi(\rho(\Frob_u))\Big|^2\Big)^{1/2}.
\end{align}

Now fix any character $\chi\in\widehat{G}-\widehat{\Gamma}$.  Let $\scrF_\chi$ be a lisse $\Qbar_\ell$-adic sheaf corresponding to the character $\chi \circ \rho \colon \pi_1(U) \to \Qbar_\ell$.  By the Grothendieck-Lefschetz trace formula, we have
\[
\sum_{u\in U(\FF_q)} \chi(\rho(\Frob_u)) = \sum_{i=0}^2 (-1)^i \Tr(\Fr | H^i_c(U_{\FFbar_q},\scrF_\chi))
\]
where $\Fr$ is the geometric Frobenius automorphism.  By Deligne's theorem, the eigenvalues of $\Fr$ acting on $H^i_c(U_{\FFbar_q},\scrF_\chi)$ are algebraic integers with absolute values $\leq q^{i/2}$ in $\CC$ (under any embedding $\Qbar_\ell\hookrightarrow \CC$).  So
\[
\Big| \sum_{u\in U(\FF_q)} \chi(\rho(\Frob_u))\Big| \leq \sum_{i=0}^2 q^{i/2} \dim H^i_c(U_{\FFbar_q},\scrF_\chi).
\]  
The sheaf $\scrF_\chi$ comes from an irreducible representation of $G$ for which $G^g$ acts non-trivially (because $\chi\not\in\widehat{\Gamma}$), so the coinvariants $(\scrF_\chi)_{\pi_1(U_{\FFbar_q})}$ are trivial.  Therefore, $H^2_c(U_{\FFbar_q},\scrF_\chi)=0$ since it is canonically isomorphic to $(\scrF_\chi)_{\pi_1(U_{\FFbar_q})}(-1).$ Since $U$ is affine and smooth, we also have $H^0_c(U_{\FFbar_q},\scrF_\chi)=0$. Therefore
\[
\Big| \sum_{u\in U(\FF_q)} \chi(\rho(\Frob_u))\Big| \leq q^{1/2} \dim H^1_c(U_{\FFbar_q},\scrF_\chi) = -q^{1/2} \chi_c(U_{\FFbar_q},\scrF_\chi)
\] 
where $\chi_c(U_{\FFbar_q},\scrF_\chi):= \sum_{i=0}^2 (-1)^i\dim H^i_c(U_{\FFbar_q},\scrF_\chi).$  By \cite{MR955052}*{\S2.3.1},
\[
\chi_c(U_{\FFbar_q},\scrF_\chi) = \chi(1)\cdot \chi_c(U_{\FFbar_q},\Qbar_\ell) = \chi(1) \big( 2 -2g + M \big)
\]
(the Swan conductors that occur are all zero by our tameness assumption on $\rho$).   Therefore
\[
\Big| \sum_{u\in U(\FF_q)} \chi(\rho(\Frob_u))\Big| \leq \chi(1)\cdot q^{1/2}(2g-2+M).
\] 
(Note that there is no contradiction if $2g-2+M <0$.  In these cases we have $\widehat{G}=\widehat{\Gamma}$.)

Returning to (\ref{E:Chebotarev character setup}), we have
\begin{align*}
\Big| |\{ u \in U(\FF_q): \rho(\Frob_u) \subseteq C \}| - \frac{|C|}{|G^g|} |U(\FF_q)| \Big| & \leq  \frac{|C|^{1/2}}{|G|^{1/2}} \Big(\sum_{\chi\in\widehat{G}-\widehat{\Gamma}} \chi(1)^2\Big)^{1/2} q^{1/2} (2g-2+M)\\
&=  \frac{|C|^{1/2}}{|G|^{1/2}} (|G|-|\Gamma|)^{1/2} q^{1/2} (2g-2+M). \qedhere
\end{align*}
\end{proof}

\subsection{Intersection with lines}\label{SS:intersection of lines}
We shall use the same set-up as \S\ref{SS:main}.  Let $k$ be a number field, and let $U$ be a non-empty open subvariety of $\PP^n_k$.  Let $\calZ$ be the Zariski closure of $\PP^n_k-U$ in $\PP^n_{\OO_k}$ (where $\PP_k^n$ is the generic fiber $\PP^n_{\OO_k}$).  We define $\calU$ to be the complement of $\calZ$ in $\PP^n_{\OO_k}$, it is an open subscheme of $\PP^n_{\OO_k}$ with generic fiber $U$.  Fix a continuous and surjective homomorphism
\[
\rho\colon \pi_1(\calU_\OO) \to G
\]
where $G$ is a finite group and $\OO$ is the ring of $S$-integers in $k$ for a fixed finite set $S\subseteq \Sigma_k$.  Let $G^g$ be the image of $\pi_1(\calU_{\kbar})$ under $\rho$, and let $K$ be the minimal extension of $k$ in $\kbar$ for which $G^g$ is the image of $\pi_1(\calU_K)$.   We have a short exact sequence
\[
1\to G^g \to G \overset{\varphi}{\to} \Gal(K/k) \to 1.
\]
For all $\p\in\Sigma_k-S$ and $u\in \calU(\FF_\p)$, we have $\varphi(\rho(\Frob_u)) \in (\p,K/k)$.

Let $\bfGr/\OO_k$ be the Grassmannian scheme $\operatorname{Grass}(1,n)$ over $\OO_k$.  For any field extension $k'$ of $k$, $\bfGr_{k'}$ is the familiar variety which parametrizes the linear $1$-dimension subvarieties (i.e., lines) of $\PP^n_{k'}$.  Let $W$ be the closed subvariety of $\bfGr_k$ such that for every algebraically closed extension $k'/k$ and line $L \in \bfGr(k')$, we have $L \not\in W(k')$ if and only if $L$ intersects $\calZ_{k'}$ only at smooth points of $\calZ_{k'}$, and transversally at each of these points.  Our interest in the variety $W$ is due to the following lemma.

\begin{lemma} \label{L:topological}
For all lines $L \in (\bfGr_k-W)(\kbar)$, the homomorphism  
\[
\pi_1(U_{\kbar} \cap L ) \to \pi_1(U_{\kbar}) \overset{\rho}{\to} G
\]
has image $G^g$.
\end{lemma}
\begin{proof}
Choosing an embedding $\kbar\hookrightarrow \CC$, it suffices to prove the lemma for an arbitrary line $L \in (\bfGr_k-W)(\CC)$ (the image of $\pi_1(U_\CC)$ under $\rho$ is still $G^g$).  The lemma is true for a generic line by Bertini's theorem, so the result follows by (topologically) deforming $L$ to a generic element in $(\bfGr_k-W)(\CC)$.
\end{proof}

Let $\calW$ be the Zariski closure of $W$ in $\bfGr$.  We now prove an equidistribution theorem for lines $L$ in $\PP^n_{\FF_\p}$ that do not lie in $\calW(\FF_p)$.  It will allow us to reduce our Hilbert irreducibility bounds to the one dimensional setting.

\begin{thm} \label{T:fibered Chebotarev}
Let $C$ be a subset of $G$ that is stable under conjugation such that $\kappa := \varphi(C)$ is a conjugacy class of $\Gal(K/k)$.    Take any prime $\p\in \Sigma_k-S$ for which $\p \nmid |G^g|$ and $(\p,K/k)=\kappa$, and any line $L\in (\bfGr_{\FF_\p}-\calW_{\FF_\p})(\FF_\p)$.  Then 
\[
|\{ u\in\calU(\FF_\p) \cap L(\FF_\p) : \rho(\Frob_u) \subseteq C \}| = \frac{1}{|\kappa|}\frac{|C|}{|G^g|} N(\p) + O_{U}\bigg(\frac{|C|^{1/2}}{|\kappa|^{1/2}}  N(\p)^{1/2} \bigg).
\]
\end{thm}

\begin{proof}
We first introduce some standard notation.  Let $k_\p$ be the completion of $k$ at the prime $\p$.  Let $\OO_\p^{\un}$ be the ring of integers in the maximal unramified extension of $k_\p^{\un}$ of $k_\p$ (in a fixed algebraic closure $\kbar_\p$).  The ring $\OO_\p^{\un}$ is a complete discrete valuation ring with residue field $\FFbar_\p$.   

By excluding a finite number of $\p\in \Sigma_k-S$ (that depend only on $\calU\subseteq \PP^n_{\OO_k}$, and hence only on $U\subseteq \PP^n_k$), we can assume that  each line $L \in (\bfGr-\calW)(\FF_\p)$ lifts to a line $\calL \in  (\bfGr-\calW)(\OO_\p)$ by Hensel's lemma.  

Let $\calD$ be the scheme theoretic intersection of $\calL$ and $\calZ_{\OO_\p}$.  It is a horizontal divisor of $\calL$ which is \'etale over $\Spec \OO_\p$.    Let $\calV$ be the $\OO_\p$-scheme  $\calL-\calD$.  Choose a point $a_0 \in \calV(\FFbar_\p)$ with a lift $a_1 \in \calV(\OO_\p^{\un})$.  By the Grothendieck specialization theorem, the natural homomorphisms
\[
 \pi_1(\calV_{k_\p^{un}}, a_1) \to \pi_1(\calV_{\OO_\p^{\un}},a_1) \leftarrow \pi_1(\calV_{\FFbar_\p}, a_0)
\]
induce an isomorphism between the prime to $p=\operatorname{char} \FF_\p$ quotients of $\pi_1(\calV_{\kbar},a_1) \xrightarrow{\sim} \pi_1(\calV_{\kbar_\p}, a_1) $ and $\pi_1(\calV_{\FFbar_\p}, a_0)$.  In the present setting, an accessible proof of Grothendieck's theorem can be found in \cite{MR1708609}*{\S4}.     Therefore the homomorphism
\begin{align*}
\pi_1(\calV_{\FFbar_\p},a_0) \to \pi_1(\calV_{\OO_\p^{\un}},a_1) &\to \pi_1(\calU,a_1) \xrightarrow{\rho} G 
\end{align*}
has the same image as $\pi_1(\calV_{\kbar_\p},a_1) \to \pi_1(\calV_{\OO_\p^{\un}},a_1) \to \pi_1(\calU,a_1) \xrightarrow{\rho} G,$ which by Lemma~\ref{L:topological} is $G^g$ (the assumption that $\p\nmid |G^g|$ is needed here).\\

Let $\rho_\p$ be the representation $\pi_1(\calV_{\FF_\p},a_0) \to \pi_1(\calU,a_1)  \xrightarrow{\rho} G$, and denote its image by $G_\p$.  We have just shown that $\rho_\p(\pi_1(\calV_{\FFbar_\p},a_0))= G^g$.   Let $d$ be the index $[G_\p:G^g]$ and let $\FF$ be the degree $d$ extension of $\FF_{\p}$.   We have a short exact sequence
\[
1\to G^g\to G_\p \xrightarrow{\varphi_\p} \Gal(\FF/\FF_\p) \to 1.
\]
Define the set $C'=C\cap G_\p$, which is stable under conjugation in $G_\p$.   For $u\in \calV(\FF_\p)\subseteq \calU(\FF_\p)$, we have $\rho(\Frob_u)\subseteq C$ if and only if $\rho_\p(\Frob_u)\subseteq C'$.  Hence
\[
|\{ u\in\calU(\FF_\p)\cap L(\FF_\p) : \rho(\Frob_u) \subseteq C \}| = |\{ u\in \calV_{\FF_\p}(\FF_\p) : \rho_\p(\Frob_u) \subseteq C'\}|.
\]
Our assumption that $\varphi(C)=\kappa$ and $(\p,K/k)=\kappa$ implies that the set $\varphi_\p(C')$ consists of just the $N(\p)$-th power Frobenius automorphism.  Therefore by Proposition~\ref{P:Chebotarev positive char}
\[
\Big||\{ u\in\calU(\FF_\p)\cap L(\FF_\p) : \rho(\Frob_u) \subseteq C \}| - \frac{|C'|}{|G^g|} |\calU(\FF_\p)| \Big|
\leq |C'|^{1/2} (1-|G^g|^{-1})^{1/2}  (2\cdot 0-2+|\calD(\FFbar_\p)|) N(\p)^{1/2},
\]
where we have used that geometrically $\rho_\p$ is at worst tamely ramified (since $\p\nmid |G^g|)$.

Since $ D \to \Spec\OO_{\p}$ is \'etale and $\calL_{\kbar} \not\in W(\kbar)$, we have $|D(\FFbar_\p)|=| D(\kbar)| \ll_U 1.$  So
\begin{align*}
&\Big||\{ u\in\calU(\FF_\p)\cap L(\FF_\p) : \rho(\Frob_u) \subseteq C \}| - \frac{|C'|}{|G^g_\p|} |\calU(\FF_\p)| \Big|\\
\leq& |C'|^{1/2} (1-|G_\p^g|^{-1})^{1/2}  (-2+| D(\kbar)|) N(\p)^{1/2} \ll_U |C'|^{1/2} N(\p)^{1/2}.
\end{align*}
The theorem follows by noting that $|C'|=|C|/|\kappa|$.
\end{proof}

For a line $\calL \not\in W(k)$, we can consider its reduction $\calL_{\FF_\p}$ in $\bfGr(\FF_\p)$ for primes $\p\in\Sigma_k$.   To apply Theorem~\ref{T:fibered Chebotarev} we need that $\calL_{\FF_\p}$ does not lie in $\calW_{\FF_\p}$.  The follow lemma controls the number of primes that have this property (this will be important later when we vary the line $\calL$).   Choose an embedding $\bfGr_k\hookrightarrow \PP^N_k$ (for example, the Pl\"ucker embedding with $N= \binom{n+1}{2}$), and let $H$ be a height on $\bfGr_k$ coming from the height on $\PP^N_k$.

\begin{lemma} \label{L:height for log B term}
For any line $\calL \in \bfGr(k)-W(k),$
\[
\sum_{\substack{\p \in \Sigma_k-S \\ \calL_{\FF_\p}\, \in \calW(\FF_\p)}} \log N(\p)   \ll_{U} \log H(\calL) + O(1)
\]
where the implied constant depends only on $U\subseteq \PP^n_k$ (and in particular not on $\calL$).  
\end{lemma}
\begin{proof} 
Fix a non-constant morphism $\phi \colon \bfGr_k \to \PP_k^1$ for which $\phi^{-1}([0:1])\supseteq W$.  By choosing a model of $\phi$ over $\OO_k$, we will have morphisms $\bfGr_{\FF_\p} \to \PP^1_{\FF_\p}$ of special fibers such that $\calW_{\FF_\p}$ lies in the fibre above $[0:1]$ for most $\p$.  Therefore
\begin{align*}
\sum_{\substack{\p \in \Sigma_k-S \\ \calL_{\FF_\p} \,\in \calW(\FF_\p)}} \log N(\p)  & \leq  \sum_{ \substack{\p \in \Sigma_k-S \\ \phi(\calL) \bmod{\p} = [0:1] \in \PP^1(\FF_\p)}} \log N(\p)  + O(1) \ll_k \log H(\phi(\calL)) + O(1)
\end{align*}
by Lemma~\ref{L:Weil height}.  Finally, note that $\log H(\phi(\calL)) \ll_\phi \log H(\calL) + O(1)$ (cf.~\cite{MR1757192}*{\S2.6}).
\end{proof}

\section{Proof of Theorem~\ref{T:Main}}  \label{S:main proof}

\subsection{Proof of Theorem~\ref{T:Main}(i)} \label{SS:Integral point version}
Fix notation as in \S\ref{SS:main} and \S\ref{SS:intersection of lines}.  Without loss of generality, we may assume that $\calU$ is an open subscheme of $\AA^n_{\OO_k}=\Spec \OO_k[x_1,\ldots, x_n]$ where we view $\AA^n_{\OO_k}$ as an open subscheme of $\PP^n_{\OO_k}$ via the map $(x_1,\ldots, x_n) \mapsto [x_1,\ldots, x_n,1]$.
Let 
\[
\calL \colon \AA^{n-1}_k \to \bfGr_k,\,\quad b \mapsto \calL_b
\]
be the morphism for which $\calL_b$ is the line defined by $x_1=b_1,\ldots, x_{n-1}=b_{n-1}$ for $b=(b_1,\ldots, b_{n-1})$.  Without loss of generality, we may assume that the image of the morphism $\calL$ does not lie in $W\subsetneq \bfGr_k$ (if not, then we can arrange this by an initial change of coordinates).  \\

We then have a disjoint union
\[
\{u\in \calU(k) \cap \OO_k^n: \norm{u}\leq B \} = \bigsqcup_{b \in \OO_k^{n-1},\, \norm{b} \leq B} \{  (b_1,\ldots, b_{n-1}, a) \in \calL_b \cap \calU(k): a\in \OO_k,\, \norm{a} \leq B \}.
\]
We first consider those $b$ for which $\calL_b\in W(k)$.  Since $W$ does not lie in the image of $\calL \colon \AA_k^{n-1} \to \bfGr_k$, we find that $\calL^{-1}(W)$ is a closed subvariety of $\AA^{n-1}_k$ of codimension $\geq 1$.   So using trivial bounds for each of these lines, we have
\begin{align*}
& \sum_{\substack{b \in \OO_k^{n-1},\, \norm{b} \leq B \\ \calL_b \in W(k)}} |\{  u=(b_1,\ldots, b_{n-1}, a) \in \calL_b \cap \calU(k): a\in\OO_k,\, \norm{a} \leq B,\, G_u \subseteq C \}| \\
 \ll_k &  \, B^{[k:\QQ]} \cdot |\{ b \in \OO_k^{n-1}: \norm{b} \leq B,\, \calL_b \in W(k) \}| \ll_U B^{[k:\QQ]} \cdot B^{[k:\QQ](n-2)} = B^{[k:\QQ](n-1)}.
\end{align*}
This gives:
\begin{align} \label{E:sum decomp}
& |\{ u \in \calU(k)\cap \OO_k^n : \norm{u}\leq B,\, G_u\subseteq C\}| + O_U(B^{[k:\QQ](n-1)})  \\
\notag =& \sum_{\substack{b \in \OO_k^{n-1},\, \norm{b} \leq B \\ \calL_b \not\in W(k)}} |\{  (b_1,\ldots, b_{n-1}, a) \in \calL_b \cap \calU(k): a\in\OO_k,\, \norm{a} \leq B,\, G_u \subseteq C \}| \\
\notag \ll_{U} &\, B^{[k:\QQ](n-1)}  \max_{\substack{b \in \OO_k^{n-1},\, \norm{b} \leq B \\ \calL_b \not\in W(k)}} |\{  u=(b_1,\ldots, b_{n-1}, a) \in \calL_b \cap \calU(k): a\in\OO_k,\, \norm{a} \leq B,\, G_u \subseteq C \}|.\\ \notag
\end{align}

Now fix any $b \in \OO_k^{n-1}$ with $\norm{b} \leq B$ for which  $\calL_b \not\in W(k)$.    Let $\calA$ be the set of $u=(b_1,\ldots, b_{n-1},a) \in \calU(k)$ with $a\in \OO_k$ for which $\norm{a}\leq B$ and $G_u\subseteq C$.  We will show that
\begin{equation}\label{E:down to a line}
|\calA| \ll_{U}  |G^g|^2 \exp\bigg(\sum_{\substack{\p\in S \text{ with $\deg(\p)=1$}\\ \text{and $N(\p)\geq |G^g|^2$}}} \frac{\log N(\p)}{N(\p)}\bigg)  B^{[k:\QQ]\delta} \log B.
\end{equation}
Applying this to (\ref{E:sum decomp}) then gives
\[
|\{ u \in \calU(k)\cap \OO_k^n : \norm{u}\leq B,\, G_u\subseteq C\}|   \ll_{U}  |G^g|^2 \exp\bigg(\sum_{\substack{\p\in S \text{ with $\deg(\p)=1$}\\ \text{and $N(\p)\geq |G^g|^2$}}} \frac{\log N(\p)}{N(\p)}\bigg)  B^{[k:\QQ](n-1+\delta)} \log B.
\]
which will complete the proof of Theorem~\ref{T:Main}(i).\\

With our fixed $b$, we will now prove (\ref{E:down to a line}).  Let $T$ be the finite set of primes $\p\in\Sigma_k$ for which either $\p$ divides $|G^g|$ or for which $\calL_b \bmod{\p} \in \calW(\FF_\p)$.  Take any $\p \in \Sigma_k-(S\cup T)$,  and let $g_\p$ be the cardinality of the image of $\calA$ under the reduction modulo $\p$ map $\OO_k^n \mapsto \FF_\p^n$.  Let $\kappa$ be the conjugacy class $(\p,K/k)$ of $\Gal(K/k)$.    By Theorem~\ref{T:fibered Chebotarev},
\begin{align*}
g_\p \leq \frac{1}{|\kappa|}\frac{|C_\kappa|}{|G^g|} N(\p) + O_{U}\Big(\frac{|C_\kappa|^{1/2}}{|\kappa|^{1/2}}  N(\p)^{1/2} \Big) &\leq \frac{1}{|\kappa|}\frac{|C_\kappa|}{|G^g|} \bigg( N(\p) + O_{U}\Big(\frac{|\kappa|^{1/2}}{|C_\kappa|^{1/2}} |G^g|  N(\p)^{1/2} \Big) \bigg)\\
& \leq \delta \Big( N(\p) + c_0 |G^g|  N(\p)^{1/2}  \Big)
\end{align*}
where $c_0\geq 1$ is a constant depending only on $U\subseteq \PP^n_k$.  By Proposition~\ref{P:specialized larger sieve} (ii), we have the bound
\begin{align} \label{E:almost there}
|\calA|& \ll_{U}  |G^g|^2 \exp\bigg(\sum_{\p\in S \cup T \text{ with $\deg(\p)=1$ and $N(\p)\geq |G^g|^2$}} \frac{\log N(\p)}{N(\p)}\bigg)  B^{[k:\QQ]\delta}.
\end{align}
If $\p$ divides $|G^g|$, then $\deg(\p)=1$ and $N(\p)\geq |G^g|^2$ cannot both hold; thus these primes do not contribute to (\ref{E:almost there}).

For any non-empty finite set $R\subseteq \Sigma_k$, we have $\exp\big(\sum_{\p\in R} \frac{\log N(\p)}{N(\p)}\big) \ll_k \sum_{\p\in R} \log N(\p)$ \cite{MR1395936}*{Corollary~2.3}.   This and Lemma~\ref{L:height for log B term} give us
\[
\exp\bigg(\sum_{\substack{\p \in \Sigma_k-S,\, \calL_b \bmod{\p} \in \calW(\FF_\p) \\N(\p)\geq |G^g|^2}} \frac{\log N(\p)}{N(\p)} \bigg) \ll_k \sum_{\substack{\p \in \Sigma_k-S\\ \calL_b \bmod{\p} \in \calW(\FF_\p)} }\log N(\p) +O(1)  \ll_U \log H(\calL_b) + O(1).
\]
Observe that $\log H(\calL_b) \ll_{\calL} \log H(b) + O(1) \ll_k \log B$ (cf.~\cite{MR1757192}*{\S2.6}).
Combining these additional bounds with (\ref{E:almost there}) gives the desired bound (\ref{E:down to a line}).

\subsection{Proof of Theorem~\ref{T:Main}(ii)} \label{SS:reduction to rational case}
We will reduce to the integral points case using the following proposition (see \cite{MR1757192}*{\S13.4}).

\begin{prop} \label{P:integral to rational} 
Let $k$ be a number field and $n$ a positive integer.  There is a constant $c_0=c_0(k,n)$ such that every point $x\in \PP^n(k)$ is representable by coordinates $a=(a_0,\ldots, a_n)\in \OO_k^{n+1}$  with 
\[
\norm{a}\leq c_0 H(x).
\]
\end{prop}

Let $f\colon \AA^{n+1}_{k}\setminus\{(0,\ldots,0)\} \to \PP^n_{k}$ be the morphism $(x_0,\ldots,x_n)\mapsto [x_0,\ldots, x_n]$.  Without loss of generality, we may assume that $U$ lies in the image of $f$.  Let $U'$ be the inverse image of $U$ under $f$; it is a non-empty open subscheme of $\AA^{n+1}_{k}$.    Define the representation
\[
\rho'\colon \pi_1(U') \to \pi_1(U)\xrightarrow{\rho} G
\]
where the first homomorphism arises from $f$.
For each $u' \in U'(k)$, we have a representation $\Gal(\kbar/k) \xrightarrow{u'_*} \pi_1(U') \to G$ whose image we denote by $G_{u'}$.  For $u' \in U'(k)$, the groups $G_{u'}$ and $G_u$ are conjugate in $G$ where $u=f(u') \in U(k)$.  By Proposition~\ref{P:integral to rational},
\begin{align*}
& |\{ u \in U(k) : H(u)\leq B,\, G_u\subseteq C\}| \leq |\{ u' \in U'(k)\cap \OO_k^{n+1} : \norm{u'}\leq c_0 B,\, G_{u'}\subseteq C\}|.
\end{align*}
By Theorem~\ref{T:Main}(i), which was proved in the previous section, this is $O_U(c (c_0 B)^{[k:\QQ](n+\delta)}\log(c_0 B))$ and hence also $O_U(cB^{[k:\QQ](n+\delta)} \log B)$.

\section{Elliptic curves} \label{S:EC}
\subsection{Set up}   
Fix a number field $k$.  Let $\pi\colon E\to U$ be an elliptic curve where $U$ is a non-empty open subvariety of $\PP^n_k$ (recall this means that $\pi$ is a proper smooth morphism whose fibers are geometrically connected curves of genus 1, together with a section $\OO$ of $\pi$).    For each point $u\in U(k)$, the fiber of $\pi$ over $u$ is an elliptic curve $E_u$ over $k$.   Let $\eta$ be the generic point of $U$; the generic fiber $E_\eta$ is an elliptic curve over the function field $k(U)$.    

Fix a geometric generic point $\bbar\eta$ of $U$ (equivalently, fix an algebraic closure $\overline{k(U)}$ of $k(U)$).  For each positive integer $m$, let $E[m]$ be the $m$-torsion subscheme of $E$.    The morphism $E[m]\to U$ is finite \'etale and as a lisse sheaf  corresponds to a $(\ZZ/m\ZZ)$-representation of $\pi_1(U,\bbar\eta)$ on the geometric generic fiber $E[m]_{\bbar\eta}=E_{\bbar\eta}[m]$.  We thus have a continuous homomorphism
\[
\rho_{E,m} \colon \pi_1(U,\bbar\eta) \to \Aut(E[m]_{\bbar\eta})\cong\GL_2(\ZZ/m\ZZ)
\]
which is uniquely defined up to an inner automorphism.  Let $\calH_{E}(m)$ be the image under $\rho_{E,m}$ of $\pi_1(U,\bbar\eta)$.  Combining all our representations together, we obtain a single continuous homomorphism
\[
\rho_{E} \colon \pi_1(U,\bbar\eta) \to \GL_2(\Zhat).
\]
Let $\calH_{E}$ be the image under $\rho_{E}$ of the groups $\pi_1(U,\bbar\eta)$.
 
 There is a unique morphism $j \colon U\to \AA^1_k$ such that $j(u)$ is the $j$-invariant of $E_u$ for all $u\in U(k)$.   Assume that $E\to U$ is \emph{non-isotrivial}; i.e., $j\colon U\to \AA^1_k$ is non-constant (equivalently, the $j$-invariant of $E_\eta$ does not belong to $k$).\\
 
 In \S\ref{SS:intro EC}, we started with an elliptic curve over $k(T_1,\ldots, T_n)$, which to avoid confusion we will call $\tilde{E}$.  Choosing a specific model, we described a closed subvariety $Z$ of $\AA^n_k:=\Spec k[T_1,\ldots, T_n]$ (whose $k$-points we denoted by $\Omega$) such that specializing our model at any $k$-point $t$ of $U:=\AA^n_k -Z$ gave an elliptic curve.  This describes an elliptic curve $E$ over $U$ whose generic fiber is the original $\tilde E$.   Theorems~\ref{T:HIT EC} and \ref{T:EC Main Q} are thus equivalent to:
 
 \begin{thm} \label{T:HIT EC final}  Fix notation as above.
 \begin{romanenum}
 \item
If $k\neq \QQ$, then 
\begin{align*}
\frac{|\{ u\in U(k): H(u)\leq B,\, \rho_{E_u}(\Gal(\kbar/k))= \calH_E \}|}{|\{ u\in U(k) : H(u)\leq B \}|} &= 1 +O\big( B^{-1/2}\log B\big) \quad \text{and}\\
\frac{|\{ u \in U(k)\cap \OO_k^n: \norm{u}\leq B,\, \rho_{E_u}(\Gal(\kbar/k))= \calH_E \}|}{|\{ u\in U(k)\cap \OO_k^n : \norm{u}\leq B \}|} &= 1  +O\big( B^{-1/2}\log B\big).
\end{align*}
\item
If $k= \QQ$, then for any $\varepsilon>0$ we have
\begin{align*} 
\frac{\big|\big\{ t\in U(\QQ) : H(t)\leq B,\, \big[\calH_E : \rho_{E_t}(\Gal(\Qbar/\QQ))\big] = r \big\}\big|}{|\{ t\in \QQ^n : H(t)\leq B \}|} &=1 +O(B^{-1/2+\varepsilon}) \quad\text{and}\\
\frac{\big|\big\{ t\in U(\QQ) \cap \ZZ^n : \norm{t}\leq B,\, \big[\calH_E : \rho_{E_t}(\Gal(\Qbar/\QQ))\big] = r \big\}\big|}{|\{ t\in \ZZ^n : \norm{t}\leq B \}|} &=1 +O(B^{-1/2+\varepsilon})
\end{align*}
where $r$ is the index of $[\calH_E,\calH_E]$ in $\calH_E\cap \SL_2(\Zhat)$.
\end{romanenum}
The implicit constants depend on $E\to U$ and $k$, and also $\varepsilon$ in (ii).
\end{thm}

We claim that it suffices to prove parts (i) and (ii) of Theorem~\ref{T:HIT EC final} only in the integral points case; we explain for part (i) only.  As in \S\ref{SS:reduction to rational case}, we define a morphism $f\colon \AA^{n+1}_{k} \to \PP^n_{k}$ by $(x_0,\ldots,x_n)\mapsto [x_0,\ldots, x_n]$.  Without loss of generality, we may assume that $U$ lies in the image of $f$.  Let $U'$ be the inverse image of $U$ under $f$; it is a non-empty open subvariety of $\AA^{n+1}_{k}$.     Base extension gives an elliptic curve $E':= E\times_U U' \to U'$.   Composing the homomorphism $\pi_1(U')\to \pi_1(U)$ coming from $f$ with the representation $\rho_E\colon \pi_1(U)\to \GL_2(\Zhat)$ gives $\rho_{E'}\colon \pi_1(U')\to \GL_2(\Zhat)$ (at least up to conjugation since we are suppressed base points everywhere). For each $u'\in U'(k)$, the curves $E'_{u'}$ and $E_{f(u')}$ are isomorphic and $\rho_{E'_{u'}}(\Gal(\kbar/k))=\calH_{E'}$ if and only if $\rho_{E_{f(u')}}(\Gal(\kbar/k))=\calH_{E}$.  Proposition~\ref{P:integral to rational} implies that 
\begin{align*}
&|\{ u \in U(k) : H(u)\leq B,\, \rho_{E_u}(\Gal(\kbar/k))\neq \calH_E \}| \\ 
\leq & |\{ u' \in U'(k)\cap \OO_k^{n+1} :  \norm{u'}\leq c_0B,\, \rho_{E_{u'}}(\Gal(\kbar/k))\neq \calH_{E'}\}|
\end{align*}
and the integral case of Theorem~\ref{T:HIT EC final}(i) then says that this is $O(B^{[k:\QQ](n+1)}\cdot B^{-1/2}\log B)$ as required.\\

For the rest of \S\ref{S:EC}, we shall thus focus on the integral points setting.  We will assume that $U$ is an open subvariety of $\AA^n_k$.

\subsection{Surjectivity modulo primes}

We first consider the Galois actions on the $m$-torsion points for a fixed $m$.  The following is an explicit form of HIT in this context; it is of the utmost importance for our application that the implicit constants in part (ii) do not depend on $m=\ell$.

\begin{prop} \label{P:EC HIT first}
\begin{romanenum}
\item
For any positive integer $m$,  we have
\begin{align*}
|\{ u \in U(k)\cap\OO_k^n : \norm{u}\leq B,\, \rho_{E_u,m}(\Gal(\kbar/k)) \neq \calH_E(m)\}| &\ll_{E,m}   B^{[k:\QQ](n-1/2)} \log B.
\end{align*}
\item
For every prime $\ell\geq 17$, we have
\begin{align*}
|\{ u \in U(k)\cap\OO_k^n : \norm{u}\leq B,\, \rho_{E_u,\ell}(\Gal(\kbar/k)) \not \supseteq \SL_2(\ZZ /\ell\ZZ)\}| &\ll_{E}  \ell^6  B^{[k:\QQ](n-1/2+O(1/\ell))} \log B
\end{align*}
where the implicit constants do not depend on $\ell$ or $B$.
\end{romanenum}
\end{prop}

Before proving the proposition, we state the following criterion for a subgroup of $\GL_2(\FF_\ell)$ to contain $\SL_2(\FF_\ell)$.

\begin{lemma}   \label{L:Serre criterion}
Let $\ell\geq 5$ be a prime.   
\begin{itemize}
\item
Let $C_1(\ell)$ be the set of $A\in \GL_2(\FF_\ell)$ for which $\tr(A)^2-4\det(A)$ is a non-zero square in $\FF_\ell$, and such that $\tr(A)\neq 0$.
\item
Let $C_2(\ell)$ be the set of $A\in \GL_2(\FF_\ell)$ for which $\tr(A)^2-4\det(A)$ is not a square in $\FF_\ell$, and such that $\tr(A)\neq 0$.
\item
Let $C_3(\ell)$ be the set of $A\in\GL_2(\FF_\ell)$ such that $u=\tr(A)^2/\det(A)$ is not $0,1,2$ or $4$, and such that $u^2-3u+1\neq 0$.
\end{itemize}
\begin{romanenum}
\item
If $G$ is a subgroup of $\GL_2(\FF_\ell)$ that contains elements from all three of the sets $C_1(\ell)$, $C_2(\ell)$ and $C_3(\ell)$, then $G$ contains $\SL_2(\FF_\ell)$.  
\item
For each $d\in \FF_\ell^\times$, we have
\[
\frac{|\{A\in C_i(\ell):\det(A)=d\}|}{|\SL_2(\FF_\ell)|} =\begin{cases}
 \frac{1}{2} +O(1/\ell)     & \text{ for $i=1,2$}, \\
1+O(1/\ell)     & \text{ for $i=3$}.
\end{cases}
\]
\end{romanenum}
\end{lemma}
\begin{proof}
Part (i) is Proposition~19 of \cite{MR0387283}.  We now consider part (ii) with a fixed $d\in \FF_\ell^\times$.  For each $t\in \FF_{\ell}$, \cite{MR2178556}*{Lemma~2.7} shows that
\[
|\{A \in \GL_2(\FF_\ell) : \det(A)=d,\, \tr(A)=t \} | = \ell^2+ \epsilon\ell \quad \text{ where } \epsilon=\legendre{t^2-4d}{\ell} \in \{-1,0,1\}.
\] 
Hence for each $c\in\FF_\ell$,
\[
|\{ A \in \GL_2(\FF_\ell) : \det(A) = d,\, \tr(A)^2/d = c\}| \leq 2\ell(\ell+1).
\]
This implies the bound for $C_3(\ell)$ and taking $c=4$ shows that we need only prove the bound for $C_1(\ell)$ or $C_2(\ell)$.  We have
\begin{align*}
|\{A \in C_1(\ell) : \det(A)=d\}| &= \ell^2 |\{t\in \FF_\ell : t^2-4d \text{ is a square in $\FF_\ell$}\}| + O(\ell^2) \\
&= \frac{1}{2} \ell^2 \cdot |\{(t,y)\in \FF_\ell^2 : t^2-y^2 =4d \}| + O(\ell^2).
\end{align*}
Since $d\neq 0$ and $\ell$ is odd, the plane curve $t^2-y^2 =4d$ in $\AA^2_{\FF_\ell}=\Spec \FF_\ell[t,y]$ is isomorphic to $\PP^1_{\FF_\ell}$ with two $\FF_\ell$-rational points removed.  Therefore, $|\{A \in C_1(\ell) : \det(A)=d\}|= \frac{1}{2}\ell^2 (\ell-1) + O(\ell^2)= \frac{1}{2}\ell^3 + O(\ell^2)$.
\end{proof}

\begin{proof}[Proof of Proposition~\ref{P:EC HIT first}]
\noindent (i) This follows from the large sieve bounds in Theorem~\ref{T:HIT large}.

\noindent (ii)  We first extend the elliptic curve $\pi\colon E \to U$ to an integral model. There is a finite set $S\subseteq \Sigma_k$, an open subscheme $\calU$ of $\AA^n_{\OO}$ over the ring $\OO$ of $S$-integers, and an elliptic curve $\calE \to \calU$ such that the generic fibers of $\calE$ and $\calU$ are $E$ and $U$, respectively, and $\pi\colon E\to U$ is the morphism on generic fibers of $\calE\to\calU$.  

Now fix a prime $\ell\geq 17$.  The representation $\rho_{E,\ell}\colon \pi_1(U)\to \GL_2(\ZZ/\ell\ZZ)$ factors through a Galois representation
\[
 \pi_1(\calU_{\OO_\ell})\to \GL_2(\ZZ/\ell\ZZ)
\]
where $\OO_\ell$ is the ring of $S_\ell$-integers with $S_\ell := S \cup \{\p\in \Sigma_k: \p | \ell\}$; note that the torsion subscheme $\calE_{\OO_\ell}[\ell]\to \calU_{\OO_\ell}$ is finite \'etale.  

 Let $\calH_{E}^g(\ell)$ denote the image under $\rho_{E,\ell}$ of $\pi_1(U_{\kbar},\bbar\eta)$, and assume that $\calH_E^g(\ell)=\SL_2(\ZZ/\ell\ZZ)$.  Let $C_1(\ell)$, $C_2(\ell)$ and $C_3(\ell)$ be the sets defined in Lemma~\ref{L:Serre criterion}.     
By Lemma~\ref{L:Serre criterion}(i), we have
\begin{align*}
&|\{u\in U(k)\cap \OO_k^n : \norm{u} \leq B, \, \rho_{E_u,\ell}(\Gal(\kbar/k))\not\supseteq \SL_2(\ZZ/\ell\ZZ)\}| \\
\leq &\sum_{i=1}^3 |\{u\in U(k)\cap \OO_k^n : \norm{u} \leq B, \, \rho_{E_u,\ell}(\Gal(\kbar/k))\subseteq \GL_2(\ZZ/\ell\ZZ)-C_i(\ell)\}|.
\end{align*}
By Theorem~\ref{T:Main},
\[
|\{u\in U(k)\cap \OO_k^n : \norm{u} \leq B, \, \rho_{E_u,\ell}(\Gal(\kbar/k))\not \supseteq \SL_2(\ZZ/\ell\ZZ)\}|  \ll_U c B^{[k:\QQ] (n-1+\delta)}\log B
\]
where 
\[
c:= |\SL_2(\ZZ/\ell\ZZ)|^2\exp\bigg(\sum_{\substack{\p\in S \text{ or }\p|\ell \\ \deg(\p)=1 \text{ and } N(\p)\geq |\SL_2(\ZZ/\ell\ZZ)|^2 }} \frac{\log N(\p)}{N(\p)} \bigg)
\]
and 
\[
\delta:= \max_{\substack{i=1,2,3\\ d \in \det(\calH_E(\ell))}} \frac{|\{A \in \GL_2(\ZZ/\ell\ZZ) - C_i(\ell) : \det(A)=d\}|}{|\SL_2(\ZZ/\ell\ZZ)|}.
\]
We obtain the desired bound by noting that $\delta = \frac{1}{2}+ O(1/\ell)$ from Lemma~\ref{L:Serre criterion}(ii) and that $c$ is less than $\ell^6 \exp\big(\sum_{\p\in S} \frac{\log N(\p)}{N(\p)} \big)\ll_S \ell^6.$

It thus remains to show that $\calH_E^g(\ell)=\SL_2(\ZZ/\ell\ZZ)$ for every prime $\ell\geq 17$.   After choosing an embedding $\kbar\hookrightarrow \CC$, we have $\calH_E^g(\ell) = \rho_{E,\ell}(\pi_1(U_{\CC}))$.  Let $X(\ell)$ be the modular curve over $\CC$ which classifies elliptic curves with a basis for the $\ell$-torsion.  There is a natural action of $\SL_2(\ZZ/\ell\ZZ)$ on $X(\ell)$ and the quotient gives a morphism $ X(\ell) \to X(1)$ where $X(1)\cong \PP^1_{\CC}$ is the $j$-line.  Now consider the quotient curve $X_E:=X(\ell)/\calH_E^g(\ell)$ and the natural morphism $f\colon X_E \to X(1)$.  There is a morphism $h\colon U_\CC\to X_E$ such that the $j$-invariant of $E_u$ is $f(h(u))$ for all $u\in U(\CC)$.  The morphism $f\circ h$, and hence $h$, is non-constant by our ongoing assumption that the $j$-invariant of $E$ is non-constant.  Since $U_\CC$ is open in $\PP^n_\CC$ and $h$ is dominant, we deduce that $X_E$ has genus 0.   For $\ell\geq 17$, there are no proper subgroups $H$ of $\SL_2(\ZZ/\ell\ZZ)$ for which $X(\ell)/H$ has genus 0 (it suffices to compute the genus of $X(\ell)/H$ for the for maximal subgroups $H$ of $\SL_2(\ZZ/\ell\ZZ)$, see \cite{MR2189500}*{Table 2.1}).
\end{proof}

The following effective version of Serre's open image theorem, due to Masser and W\"ustholz, allows us to effectively bound the primes $\ell$ that have to be considered.

\begin{thm}[Masser-W\"ustholz \cite{MR1209248}] \label{T:M-W}
There are absolute constants $c>0$ and $\gamma\geq 0$ with the following properties.  Suppose $E$ is an elliptic curve of Weil height\footnote{i.e., the absolute logarithmic height of the $j$-invariant of $E$} $h$ defined over a number field $k$ of degree $d,$ and assume $E$ has no complex multiplication over $\kbar$.  If $\ell> c(\max\{d,h\})^\gamma$, then $\rho_{E,\ell}(\Gal(\kbar/k))\supseteq \SL_2(\ZZ/\ell\ZZ)$.
\end{thm}

Combining Masser and W\"ustholz's theorem with our explicit HIT bounds gives the following proposition.  

\begin{prop} \label{P:EC HIT all primes}
For every $\varepsilon>0$, we have
\begin{align*}
\frac{|\{ u \in U(k)\cap\OO_k^n : \norm{u}\leq B,\, \rho_{E_u,\ell}(\Gal(\kbar/k)) \supseteq \SL_2(\ZZ /\ell\ZZ) \text{ for all $\ell\geq 17$}\}|}{|\{ u \in U(k)\cap\OO_k^n : \norm{u}\leq B\}|} & =1 +O_{E,\varepsilon}\Big(\frac{1}{B^{[k:\QQ]/2 - \varepsilon}}\Big).
\end{align*}
\end{prop}
\begin{proof}
Recall that there is a morphism $j_E\colon U \to \AA^1_k$ such that for each $u\in U(k)$, the $j$-invariant of $E_u$ is $j_E(u)$.   Now take any $u\in U(k)\cap \OO_k^n$ with $\norm{u} \leq B$.     We have 
\[
\log H(j(E_u))=\log H(j_E(u))\ll \log H(u) \ll \log \norm{u} \leq \log B
\] 
where the implicit constants do not depend on $u\in U(k)\cap \OO_k^n$.   So by Theorem~\ref{T:M-W} if $E_u$ is non-CM, then $\rho_{E_u,\ell}(\Gal(\kbar/k)) \supseteq \SL_2(\ZZ /\ell\ZZ)$ for all $\ell \geq C (\log B)^\gamma$ where $\gamma\geq 0$ is an absolute constant and $C$ is a constant that depends on $E$ and $k$.  Therefore,
\begin{align*}
& \, |\{ u\in U(k)\cap \OO_k^n : \norm{u} \leq B,\, \rho_{E_u,\ell}(\Gal(\kbar/k)) \not\supseteq \SL_2(\ZZ /\ell\ZZ) \text{ for some $\ell\geq 17$}\}|\\
\leq &\, \sum_{17\leq \ell \leq C(\log B)^\gamma}   |\{ u\in U(k)\cap \OO_k^n : \norm{u} \leq B,\, \rho_{E_u,\ell}(\Gal(\kbar/k)) \not\supseteq \SL_2(\ZZ /\ell\ZZ)\}|
\end{align*}
(note that if $E_u$ has complex multiplication then $\rho_{E_u,\ell}(\Gal(\kbar/k)) \not\supseteq \SL_2(\ZZ /\ell\ZZ)$ for all $\ell\geq 17$).
By Theorem~\ref{P:EC HIT first},
\begin{align} \label{E:EC HIT all primes}
& \, |\{ u\in U(k)\cap \OO_k^n : \norm{u} \leq B,\, \rho_{E_u,\ell}(\Gal(\kbar/k)) \not\supseteq \SL_2(\ZZ /\ell\ZZ) \text{ for some $\ell\geq 17$}\}|\\
\notag &\ll_{E,\varepsilon} \, \sum_{17\leq \ell \leq C(\log B)^\gamma}   \ell^6  B^{[k:\QQ](n-1/2+\varepsilon)}\log B.
\end{align}
We have used part (ii) of Theorem~\ref{P:EC HIT first} for all sufficiently large $\ell$ (how large depends on $\varepsilon$ but not on $B$) and Theorem~\ref{P:EC HIT first}(i) is used for the finitely many excluded primes.  So (\ref{E:EC HIT all primes}) is $O(B^{[k:\QQ](n-1/2+\varepsilon)}(\log B)^{6\gamma+1})$, and the proposition follows from (\ref{E:comparison}) and a readjustment of $\varepsilon$.
\end{proof}

The following group theoretic lemma justifies our focus on the Galois images arising from $\ell$-torsion.   We will apply it later with $\calH$ equal to $[\calH_E,\calH_E]$.

\begin{lemma} \label{L:Openness criterion}
Let $\calH$ be an open subgroup of $\SL_{2}(\Zhat)$, and let $G$ be a closed subgroup of $\calH$.  For each positive integer $m$, let $\calH(m)$ and $G(m)$ be the images under the reduction modulo $m$ map $\SL_{2}(\Zhat)\to\SL_{2}(\ZZ/m\ZZ)$ of $\calH$ and $G$, respectively.  Then there exists a positive integer $M$ (divisible only by those primes $\ell$ for which $\calH(\ell)\neq \SL_2(\ZZ/\ell\ZZ)$ or $\ell\leq 5$) such that $G=\calH$ if and only if $G(M)=\calH(M)$ and $G(\ell)=\SL_{2}(\ZZ/\ell\ZZ)$ for all primes $\ell\nmid M$.
\end{lemma}
\begin{proof} 
Let $\calH_m$ and $G_m$ be the image of $\calH$ and $G$, respectively, in $\prod_{\ell | m} \SL_2(\ZZ_\ell)$.  

Let $M_0$ be a positive integer divisible by 2, 3, 5 and by the primes for which $\calH(\ell)\neq \SL_2(\ZZ/\ell\ZZ)$.   The \emph{Frattini subgroup} $\Phi(\calH_{M_0})$ of $\calH_{M_0}$ is the intersection of the maximal closed subgroups of $\calH_{M_0}$.  Since $\calH$ is open in $\SL_2(\Zhat)$, the group $\calH_{M_0}$ contains a normal and open subgroup of the form $\prod_{\ell | M_0}\calS_{\ell^{e(\ell)}}$ for some $e(\ell)\geq 1$, where $\calS_{\ell^{e(\ell)}}:= \{A \in \SL_2(\ZZ_\ell) : A \equiv I \pmod{\ell^{e(\ell)}}\}$ .   The groups $ \{A \in \SL_2(\ZZ_\ell) : A \equiv I \pmod{\ell^{e(\ell)}}\}$ are pro-$\ell$ and are finitely generated as topological groups.   Therefore by \cite[10.6 Prop.]{MR1757192}, $\Phi(\calH_{M_0})$ is an open normal subgroup of $\calH_{M_0}$.   Choose a positive integer $M$ with the same prime divisors as $M_0$ such that $\Phi(\calH_{M_0})\supseteq \prod_{\ell^e \parallel M} \calS_{\ell^e}$; this will be our desired $M$.   Observe that if $G(M)=\calH(M)$, then $G_M = \calH_M$.

Consider a prime $\ell\nmid M_0$.  By \cite{MR1757192}*{IV-23 Lemma 3}, the assumption $G(\ell)=\calH(\ell)=\SL_2(\ZZ/\ell\ZZ)$ implies that $G_\ell =\calH_\ell = \SL_2(\ZZ_\ell)$.  

We may view $G$ and $\calH$ as subgroups of $\calH_M \times \prod_{\ell \nmid M} \SL_2(\ZZ_\ell)$.  We have seen that the projection of $G$ onto the $\calH_M$ and $\SL_2(\ZZ_\ell)$ factors is surjective.  We now show that these factors have no common non-abelian simple groups in their composition series.    For $\ell \nmid M$ (in particular $\ell\geq 5$), the only non-abelian simple group occurring in a composition series of $\SL_2(\ZZ_\ell)$ is $\SL_2(\FF_\ell)/\{\pm I\}$.   Also $\SL_2(\ZZ_\ell)$ with $\ell\geq 5$ has no non-trivial abelian quotients (cf. \cite{Zywina-Maximal}*{Lemma A.1}).   None of the groups $\SL_2(\FF_\ell)/\{\pm I\}$ ($\ell\nmid M$) occur in a composition series of $\calH_M$ (this follows from the calculation of ``$\textrm{Occ}(\SL_2(\ZZ_\ell))$'' in \cite[IV-25]{MR1484415}).    Using Goursat's lemma, we deduce the equality $G=\calH_M \times \prod_{\ell \nmid M} \SL_2(\ZZ_\ell)$ (for example, see \cite{Zywina-Maximal}*{Lemma A.4} where it is stated only for finite groups but it immediately extends to profinite groups); since $\calH$ lies between these two groups, we deduce that $G=\calH$.
\end{proof}

\subsection{Abelian quotients and cyclotomic fields}

We now state a special version of HIT involving the cyclotomic extension of $k$.  We will need this proposition in future work, so we also include a rational point version.

\begin{prop} \label{P:cyclotomic HIT}
Let $k$ be any number field \emph{except} $\QQ$.   Fix a non-empty open subvariety $U$ of $\PP^n_k$ and a surjective continuous homomorphism $\rho\colon \pi_1(U) \to G$
where $G$ is a finite abelian group.   Let $G^c$ be the image of $\pi_1(U_{k^\cyc})$ under $\rho$.   For each $u\in U(k)$, let $\rho_u$ be the composition $\Gal(\kbar/k)=\pi_1(\Spec k) \xrightarrow{u_*} \pi_1(U)\xrightarrow{\rho} G$.   Then 
\[
\frac{|\{ u \in U(k) : H(u)\leq B,\, \rho_u(\Gal(\kbar/k^\cyc)) = G^c\}|}{|\{ u \in U(k) : H(u)\leq B\}|} = 1 + O\Big( \frac{\log B}{B^{1/2}} \Big).  
\]
Assume further that $U$ is an open subvariety of $\AA^n_{k}.$  Then 
\[
\frac{|\{ u \in U(k)\cap \OO_k^n : \norm{u}\leq B,\,  \rho_u(\Gal(\kbar/k^\cyc)) = G^c \}|}{|\{ u \in U(k)\cap \OO_k^n : \norm{u}\leq B\}|} = 1 + O\Big( \frac{\log B}{B^{1/2}} \Big).  
\]
The implicit constants do not depend on $B$.
\end{prop}

Since $\QQ^\cyc$ is the maximal abelian extension of $\QQ$, Proposition~\ref{P:cyclotomic HIT} fails for $k=\QQ$ and $G^c\neq 1$.  The proof of the proposition is based on the following simple lemma.  Since we are working with an abelian group $G$, the Frobenius conjugacy classes are actually well-defined elements.  
\begin{lemma} \label{L:Frob ind}
Let $p$ be a rational prime that splits completely in $k$ and let $L$ be a finite abelian extension of $\QQ$ that is unramified at $p$.
Choose any prime $\p$ of $\OO_k$ lying over $p$.  Then the automorphism $(\p, Lk/k) \in \Gal(Lk/k)$ does not depend on the choice of $\p$ dividing $p$.
\end{lemma}
\begin{proof}
Our assumptions assure that $p$ is unramified in $Lk$.  Restriction to $L$ defines an injective homomorphism $\Gal(Lk/k) \hookrightarrow \Gal(L/\QQ)$.  We claim that $(\p,Lk/k) |_L = (p,L/\QQ)$ from which the lemma would follow immediately.  Define $\sigma:=(\p, Lk/k)$ and fix a prime $\mathfrak{P}$ of $\OO_{Lk}$ lying over $\p$.  Then $\sigma(\mathfrak{P})=\mathfrak{P}$ and $\sigma$ induces the $p$-th power Frobenius automorphism on $\FF_{\mathfrak{P}}$ (since $p=N(\p)$).   The restriction $\sigma|_L$ stabilizes the prime $\p':=\mathfrak{P}\cap \OO_L$ of $\OO_L$ and induces the $p$-th power Frobenius automorphism on $\FF_{\p'}$.  Therefore, $\sigma|_L = (p,L/\QQ)$ as claimed.
\end{proof}

\begin{proof}[Proof of Proposition~\ref{P:cyclotomic HIT}]
A similar argument to that in \S\ref{SS:reduction to rational case} shows that the rational point version is a consequence of the integral point version, so we need only prove the second statement.  Set $d=[k:\QQ]$.  As usual, define $G^g=\rho(\pi_1(U_{\kbar}))$.   If $G^g=1$, then the proposition is easy ($\rho$ factors through $\Gal(\kbar/k)$ and equals  $\rho_u$ for each $u\in U(k)$).  So we may assume that $G^g\neq 1$.  Since $G^g\subseteq G^c$, this also implies that $G^c\neq 1$.

For a fixed $u\in U(k) \cap \OO_k^n$, we certainly have $\rho_u(\Gal(\kbar/k^\cyc))\subseteq G^c$.   If this is not an equality, then $\widetilde{\rho}_u(\Gal(\kbar/k^{\cyc}))=1$ where $\widetilde{\rho}$ is the representation $\pi_1(U) \xrightarrow{\rho} G \twoheadrightarrow G/H$ for some proper subgroup $H$ of $G^c$.   Thus by (\ref{E:comparison}) it suffices to show that
\[
|\{ u \in U(k)\cap \OO_k^n : \norm{u}\leq B,\, \rho_u(\Gal(\kbar/k^\cyc)) = 1\}| \ll  B^{nd-1/2} \log B.
\]
Define the set 
\[
\calA = \{ u \in U(k) \cap \OO_k^n : \norm{u} \leq B,\, \rho_u(\Gal(\kbar/k^\cyc))=1 \}
\]
for a fixed real number $B\geq 2$.  Choose an open subscheme $\calU$ of $\AA^n_{\OO_k}$ with generic fiber $U$.  Fix a finite set $S\subseteq \Sigma_k$ for which  $\rho$ factors through a homomorphism $\pi_1(\calU_{\OO})\to G $, which we shall also denote by $\rho$, where $\OO$ is the ring of $S$-integers in $k$.  

There is a finite Galois extension $K/\QQ$ such that $K\supseteq k$ and $\rho(\pi_1(U_{K}))= G^g$.  Fix a prime $p$ that splits completely in $K$ and is not divisible by any prime in $S$.  Then for a prime $\p$ of $\OO_k$ dividing $p$ and an element $C\in G^g$, we have
\[
 |\{u\in \calU(\FF_{\p}) : \rho(\Frob_u) = C \}| = \frac{1}{|G^g|} N(\p)^n + O(N(\p)^{n-1/2}),
\]
where the implicit constant depends on $\rho$ and $K$ (this follows from Deligne's theorem and the bounds in \cite{MR0506272}).  Let $\p_1,\ldots, \p_d$ be the primes of $\OO_k$ dividing $p$.  Define the sets 
\[
B_p = \Big\{ (u_1,\ldots,u_d) \in \prod_{i=1}^d \mathcal{U}(\FF_{\p_i}) : \rho(\Frob_{u_i}) \in G^g \text{ is independent of $i$} \Big\}
\]
and $C_p =\big(\prod_{i=1}^d \FF_{\p_i}^n\big) \setminus \big(\prod_{i=1}^d  \calU(\FF_{\p_i})\big)$.   We then have $|C_p| = O(p^{dn-1})$ and
\[
|B_p| = |G^g| \Big(\dfrac{1}{|G_g|} p^n + O(p^{n-1/2}) \Big)^d = \frac{1}{|G^g|^{d-1}} p^{dn} + O(p^{dn-1/2})
\]
(we have used that $N(\p_i)=p$ since $p$ splits completely in $k$).
So using our assumption that $d>1$ (i.e., $k\neq \QQ$) and $G^g\neq 1$, we find that $|B_p \cup C_p| \leq \frac{1}{2} p^{dn} + O(p^{dn-1/2})$.

Take any $u\in \calA$.   The Chinese remainder theorem gives an isomorphism 
\begin{equation} \label{E:CRT}
\OO_k^n/p\OO_k^n = \prod_{i=1}^d (\OO_k/\p_i \OO_k)^n=\prod_{i=1}^d \FF_{\p_i}^n, 
\end{equation}
so we may identify $u \pmod p$ with the tuple $(u_1,\ldots, u_d) \in \prod_{i=1}^d \FF_{\p_i}^n$.  Suppose $u\pmod p$ does not belong to $C_p$, i.e., $u_i \in \calU(\FF_{\p_i})$ for all $i$.   Then $\rho_u$ is unramified at each $\p_i$ and $\rho_u(\Frob_{\p_i})=\rho(\Frob_{u_i})$.  The condition $\rho_u(\Gal(\kbar/k^\cyc))=1$ implies that there is a finite cyclotomic extension $L/\QQ$ unramified at $p$ such that $\rho_u(\Gal(\kbar/Lk))=1.$   By Lemma~\ref{L:Frob ind}, we deduce that
\[
\rho(\Frob_{u_i})=\rho_u(\Frob_{\p_i})=\rho_u(\Frob_{\p_j})= \rho(\Frob_{u_j})
\]
for all $i,j \in \{1,\ldots, d\}$.   So using the isomorphism (\ref{E:CRT}), we find that image of $\mathcal{A}$ modulo $p$ lies in $B_p\cup C_p$ and hence has cardinality at most $\frac{1}{2} p^{dn} + O(p^{dn-1/2})$.

We can now apply the large sieve to obtain a bound for $\calA$.  Using the large sieve as in \cite[12.1]{MR1757192} (with $K=\QQ$, $\Lambda=\OO_k^n$ with norm $\norm{\cdot}$, and $Q=B^{1/2}$) gives the bound
\[
|\calA | \ll B^{nd}/L
\]
where $L= \sum_{p \leq B^{1/2}, \, p \in \mathcal{P}} (1 + O(p^{-1/2}))$ and $\mathcal{P}$ is the set of primes $p$ that are completely split in $K$ and are not divisible by any primes in $S$.  Since $\mathcal{P}$ has positive density, we have $L\gg B^{1/2}/\log(B^{1/2})$ for sufficiently large $B$.  Therefore, $|\calA| \ll B^{nd-1/2} \log B$.
\end{proof}

\subsection{Final steps}
\begin{prop} \label{T:HIT EC 2b}
  
\begin{romanenum}
\item
For any $\varepsilon > 0$,
\[
\frac{\big|\big\{ u\in U(k)\cap\OO_k^n : \norm{u}\leq B,\, \rho_{E_u}(\Gal(\kbar/k^\ab))= [\calH_E,\calH_E] \big\}\big|}{\big|\big\{ u\in U(k)\cap\OO_k^n : \norm{u}\leq B\big\}\big|}  = 1 + O_{E,\varepsilon}\Big(\frac{1}{B^{[k:\QQ]/2 - \varepsilon}}\Big).
\]
\item
If $k\neq \QQ$, then
\[
\frac{\big|\big\{ u\in U(k)\cap\OO_k^n : \norm{u}\leq B,\, \rho_{E_u}(\Gal(\kbar/k^\cyc)) =\calH_E\cap \SL_2(\Zhat) \big\}\big|}{\big|\big\{ u\in U(k)\cap\OO_k^n : \norm{u}\leq B \big\}\big|}   = 1 + O_{E}\Big(\frac{\log B}{B^{1/2}}\Big).
\]
\end{romanenum}
\end{prop}
\begin{proof}
For $u\in U(k)$, the commutator of $\rho_{E_u}(\Gal(\kbar/k))$ is  $\rho_{E_u}(\Gal(\kbar/k^\ab))$.  Since $\rho_{E_u}(\Gal(\kbar/k))\subseteq \calH_E$, we find that $\rho_{E_u}(\Gal(\kbar/k^\ab))$ is a closed subgroup of $[\calH_E,\calH_E]$.  Since $[\calH_E,\calH_E]$ is an open subgroup of $\SL_2(\Zhat)$, there is a corresponding integer $M$ as in Lemma~\ref{L:Openness criterion}; we may assume $M$ is divisible by all primes $\ell < 17$.  With this choice of $M$, 
\begin{align} 
\notag & \frac{|\{u\in U(k)\cap\OO_k^n : \norm{u}\leq B,\, \rho_{E_u}(\Gal(\kbar/k^\ab)) \neq [\calH_E,\calH_E]\}|}{|\{u\in U(k)\cap\OO_k^n : \norm{u}\leq B\}|} \\
\label{EE: M term} \leq & \frac{|\{ u\in U(k)\cap\OO_k^n : \norm{u}\leq B, \rho_{E_u,M}(\Gal(\kbar/k^\ab)) \neq [\calH_E(M),\calH_E(M)]\}|}{|\{u\in U(k)\cap\OO_k^n : \norm{u}\leq B\}|}\\
\label{EE: l term} + & \frac{|\{ u\in U(k)\cap\OO_k^n : \norm{u}\leq B, \rho_{E_u,\ell}(\Gal(\kbar/k^\ab)) \neq \SL_2(\ZZ/\ell\ZZ) \text{ for some $\ell\nmid M$}\}|}{|\{ u\in U(k)\cap\OO_k^n : \norm{u}\leq B\}|}
\end{align}
If $\rho_{E_u, M}(\Gal(\kbar/k))=\calH_E(M)$, then $\rho_{E_u,M}(\Gal(\kbar/k^\ab)) = [\calH_E(M),\calH_E(M)]$. Thus (\ref{EE: M term}) is $O(B^{-[k:\QQ]/2}\log B)$ by Proposition~\ref{P:EC HIT first}.
For $\ell \nmid M$ (and in particular, $\ell\geq 5$), the group $\SL_2(\ZZ/\ell\ZZ)$ is its own commutator subgroup, so $\rho_{E_u,\ell}(\Gal(\kbar/k^\ab)) = \SL_2(\ZZ/\ell\ZZ)$ if and only if $\rho_{E_u,\ell}(\Gal(\kbar/k)) \supseteq \SL_2(\ZZ/\ell\ZZ)$. Thus by Proposition~\ref{P:EC HIT all primes}, the term (\ref{EE: l term}) is $O(B^{-[k:\QQ]/2 + \varepsilon})$.  Part (i) follows immediately.\\

We now consider (ii), so take $k\neq \QQ$.  Define the group $G=\calH_E\cap \SL_2(\Zhat)$.    The representation $\det\circ \rho_E$ factors through the cyclotomic character $\Gal(\kbar/k) \to \Zhat^\times$, so $\rho_E(U_{k^\cyc})=G$ and $\rho_{E_u}(\Gal(\kbar/k^\cyc))$ is a closed subgroup of $G$ for all $u\in U(k)$.

The group $[\calH_E,\calH_E]$ is a normal subgroup of finite index in $G$, so there is an integer $m$ such that reduction modulo $m$ gives an isomorphism
\[
G/[\calH_E,\calH_E] \xrightarrow{\sim} G(m)/[\calH_E(m),\calH_E(m)].
\]
Define $\widetilde{\rho} \colon \pi_1(U) \to \calH_E(m)/[\calH_E(m),\calH_E(m)]$ to be the composition of $\rho_{E,m}$ with the obvious quotient map.  The image of $\pi_1(U_{k^\cyc})$ under $\widetilde\rho$ is $G^c:=G(m)/[\calH_E(m),\calH_E(m)]$.   For each $u\in U(k)$, let $\widetilde{\rho}_u$ be the composition of $\rho_{E_u,m}$ with the quotient map $\calH_E(m) \twoheadrightarrow \calH_E(m)/[\calH_E(m),\calH_E(m)]$. By Proposition~\ref{P:cyclotomic HIT} and our assumption $k\neq \QQ$, we have
\begin{equation} \label{E:ab quotient}
\frac{|\{ u \in U(k)\cap \OO_k^n : \norm{u}\leq B,\,  \widetilde\rho_u(\Gal(\kbar/k^\cyc)) = G^c \}|}{|\{ u \in U(k)\cap \OO_k^n : \norm{u}\leq B\}|} = 1 + O\Big( \frac{\log B}{B^{1/2}} \Big).  
\end{equation}

If for $u \in U(k)\cap \OO_k^n$ we have  $\rho_{E_u}(\Gal(\kbar/k^\ab))= [\calH_E,\calH_E]$ and $\widetilde\rho_u(\Gal(\kbar/k^\cyc)) = G^c$, then $\rho_{E_u}(\Gal(\kbar/k^\cyc))$ equals  $G= \calH_E \cap \SL_2(\Zhat)$. So (ii) follows from (i) and (\ref{E:ab quotient}).
\end{proof}

\begin{proof}[Proof of Theorem~\ref{T:HIT EC final}]
As remarked in the comments following the statement of Theorem~\ref{T:HIT EC final}, it suffices to prove the integral point versions.

Since $\det\circ \rho_E\colon \pi_1(U)\to \Zhat^\times$ factors through the cyclotomic character $\Gal(\kbar/k)\to \Zhat^\times$, we find that 
\[
[\calH_E: \rho_{E_u}(\Gal(\kbar/k))]= [\calH_E\cap \SL_2(\Zhat): \rho_{E_u}(\Gal(\kbar/k^\cyc))]
\]
for all $u\in U(k)$.  If $k\neq \QQ$, then the integral point version of Theorem~\ref{T:HIT EC final}(i) is equivalent to Theorem~\ref{T:HIT EC 2b}(ii). Now suppose $k=\QQ$.  By the Kronecker-Weber theorem $\QQ^\ab=\QQ^\cyc$, so $\rho_{E_u}(\Gal(\Qbar/\QQ^\ab))\subseteq  [\calH_E,\calH_E]$ for all $u\in U(\QQ)$.  Thus
\begin{align*}
[\calH_E: \rho_{E_u}(\Gal(\Qbar/\QQ))] &= [\calH_E\cap \SL_2(\Zhat): \rho_{E_u}(\Gal(\Qbar/\QQ^\ab))]\\
&= \big[\calH_E\cap \SL_2(\Zhat): [\calH_E,\calH_E] \big] \cdot \big[  [\calH_E,\calH_E]: \rho_{E_u}(\Gal(\Qbar/\QQ^\ab))\big]\\
&= r\cdot \big[  [\calH_E,\calH_E]: \rho_{E_u}(\Gal(\Qbar/\QQ^\ab))\big].
\end{align*}
The integral point version of Theorem~\ref{T:HIT EC final}(ii) follows from Theorem~\ref{T:HIT EC 2b}(i)
\end{proof}

\bibliographystyle{plain}

\end{document}